  \title{Random graphs from a weighted\\ minor-closed class}
  \date{09 October 2012}

  \documentclass[11pt]{article}
  \usepackage{amsmath}
  \usepackage{amssymb}
  \usepackage{latexsym}
  \usepackage{graphicx}

  \newenvironment{proof}{\noindent{\bf Proof} \hspace{.1in}}{\hspace*{\fill}$\Box$}
  \newenvironment{proofof}[1]{%
  \noindent {\bf Proof of #1}}%
  {\hspace*{\fill}$\Box$}

  \newtheorem{theorem}{Theorem}[section]
  \newtheorem{lemma} [theorem] {Lemma}
  \newtheorem{corollary} [theorem] {Corollary}
  \newtheorem{prop} [theorem] {Proposition}

\def\Pr{\mbox{{\rm Pr}}\,}
\let\eps=\epsilon
\def\enddiscard{}
\long\def\discard#1\enddiscard{}

  \def\Po{\mbox{\rm Po}}

  \newcommand{\E}{{\mathbb E}}
  \newcommand{\pr}{{\mathbb P}}
  
  \newcommand{\Bigc}{{\rm Big}}
  \newcommand{\bigc}{{\rm big}}

  \newcommand{\inu}{\in_{u}}
  
  \newcommand{\whp}{whp }

  \newcommand{\gs}{{\mathcal G}^{S}}

  \newcommand{\cp}{{\mathcal P}}

  \newcommand{\ca}{{\mathcal A}}
  \newcommand{\cb}{{\mathcal B}}
  \newcommand{\cc}{{\mathcal C}}
  \newcommand{\cd}{{\mathcal D}}
  
  \newcommand{\cf}{{\mathcal F}}
  \newcommand{\cg}{{\mathcal G}}
  \newcommand{\ct}{{\mathcal T}}
  \newcommand{\cu}{{\mathcal U}}

  \newcommand{\aut}{\mbox{\small aut}}
  
  \newcommand{\core}{2{\rm -core}}
   \newcommand{\ex}{{\rm Ex}}

  \newcommand{\Frag}{{\rm Frag}}
  \newcommand{\frag}{{\rm frag}}

\begin{document}

  \author{Colin McDiarmid \\ Department of Statistics \\ Oxford University}

  \maketitle

  \begin{abstract}
  There has been much recent interest in random graphs sampled uniformly from the
  $n$-vertex graphs in a suitable
  minor-closed class, such as the class of all planar graphs.
  Here we use combinatorial and probabilistic methods to investigate a more general model.
  We consider random graphs from a `well-behaved' class of graphs:
  examples of such classes include all minor-closed classes of graphs with 2-connected
  excluded minors (such as forests, series-parallel graphs and planar graphs),
  the class of graphs embeddable on any given surface, and the class of graphs with at most $k$ vertex-disjoint cycles.
  Also, we give weights to edges and components to specify 
  probabilities, so that our
  random graphs correspond to the random cluster model, appropriately conditioned.

  We find that earlier results extend naturally in both directions,
  to general well-behaved classes of graphs, and to the weighted framework,
  for example results concerning the probability of a random graph being connected;
  and we also give results on the 2-core which are new even for the uniform (unweighted) case.
  \end{abstract}



\section{Introduction}
\label{sec.intro}

  Given a class $\ca$ of graphs (always assumed to be closed under isomorphism),
  let $\ca_n$ denote the set of graphs in $\ca$ on the vertex set $[n]:= \{1,\ldots,n\}$.
  There has been much recent interest in properties of the random graph $R_n$
  sampled uniformly from $\ca_n$,
  when $\ca$ is a suitable `structured' class of graphs such as the class of all planar graphs.
  
  Analytic methods, based on generating functions and singularity analysis,
  have over recent years been extended dramatically to handle more and more complicated classes of graphs:
  the work on planar graphs by Gim\'enez and Noy~\cite{gn09a} (see also~\cite{bgw02,gn09b})
  was a breakthrough, and very recently graphs embeddable on any given surface have been handled~\cite{bg09,cfgmn11}.
  See also~\cite{ggnw07} for graphs with no minor isomorphic to the complete bipartite graph $K_{3,3}$.
  Analytic work in papers such as those just mentioned much extends earlier combinatorial and probabilistic
  investigations, as for example in~\cite{gm04,gmsw05,gmsw07,cmcd08,cmcd09,cb08,msw05,msw06}.  For further recent related work
  (appearing in 2010 or later) see for example \cite{bnw09,bfkv11,cfgn10,dowden10,dowden11,dgn2011,dgn10,dgn11a,dgn11b,dgnps12,fp11,kl10,kmcd11,km08b,km11,cmcd11,ps10,ps11}.
  
  There is natural interest also in the case when $\ca$ is {\em any} minor-closed class of graphs,
  or for example any such class with 2-connected excluded minors; and for such investigations we still need 
  combinatorial and probabilistic methods.
  Here we use such methods, building in particular on~\cite{cmcd09}, and consider a more general model:
  we investigate random graphs from a suitable weighted class of graphs,
  where `suitable' includes all the usual suspects and more, and `weighted' is described below,
  corresponding to the random clusster model.


\subsection{The model}
\label{subsec.model}

  To introduce the model, recall that in the classical binomial random graph $G_{n,p}$ on the vertex set $[n]$,
  the ${n \choose 2}$ possible edges are included independently with probability $p$,
  where $0<p<1$, see for example~\cite{bollobas,jlr00}.
  Assuming that $\ca_n$ is non-empty, for each $H \in \ca_n$ we have
  \[
    \pr(G_{n,p}=H | G_{n,p} \in \ca ) =
    \frac{p^{e(H)}(1-p)^{{n \choose 2}-e(H)}}
    {\sum_{G \in \ca_n}p^{e(G)}(1-p)^{{n \choose 2}-e(G)}}
  =
   \frac{\lambda^{e(H)}}
    {\sum_{G \in \ca_n}\lambda^{e(G)}}
\]
  where $\lambda=p/(1-p)$.  Here $e(G)$ denotes the number of edges in $G$: we will use
  $v(G)$ similarly to denote the number of vertices in $G$.

  Now consider the more general random-cluster model, see for example ~\cite{grimmett06},
  where we are also given a parameter $\nu>0$; 
  and the random graph $R_n$ takes as values the graphs $H$ on $[n]$, with 
\[
  \pr(R_n=H) \propto p^{e(H)}(1-p)^{{n \choose 2}-e(H)} \nu^{\kappa(H)}.
\]
  Here $\kappa(H)$ denotes the number of components of $H$.
  For each $H \in \ca_n$ we have
  \[
    \pr(R_n=H \, | \, R_n \in \ca )
  =
   \frac{\lambda^{e(H)}\nu^{\kappa(H)}}
    {\sum_{G \in \ca_n}\lambda^{e(G)} \nu^{\kappa(G)}}.
\]


  This is the distribution on which we shall focus.
  Thus the distribution of our random graphs in $\ca$ is as follows.
  Given {\em edge-parameter} $\lambda>0$ and {\em component-parameter} $\nu>0$,
  we let the \emph{weighting} $\tau$ be the pair $(\lambda,\nu)$.  For each graph~$G$ we let 
  $\tau(G) = \lambda^{e(G)} \nu^{\kappa(G)}$;
  and for each finite set $\cb$ of graphs we denote
  $\sum_{G \in \cb} \tau(G)$ by $\tau(\cb)$.
  We write $R \in_{\tau} \cb$ to indicate that $R$ is a random graph which takes values in
  $\cb$ with
\[ \pr(R=H) = \frac{\tau(H)}{\tau(\cb)}  \]
  and we call $R$ a $\tau$-\emph{weighted random graph from} $\cb$.
  Given a fixed class $\ca$ of graphs, we write
  $R_n \in_{\tau} \ca$ to indicate that $R_n$ has the distribution of $R \in_{\tau} \ca_n$
  (when $\ca_n$ is non-empty).
  
  Our aim is to investigate the behaviour of the $\tau$-weighted random graph $R_n$ for a suitable graph class $\ca$,
  with a fixed $\tau$, for large $n$.
  When $\lambda=\nu=1$ of course we are back to random graphs sampled uniformly.
  Let us write $R_n \in_u \ca$ 
  to indicate that $R_n$ is uniformly distributed over $\ca_n$, as introduced in~\cite{msw06}
  (and perhaps earlier).
  

  Analytic methods for graph problems often involve generating functions with a variable $x$ for vertices
  and a variable $y$ for edges (and sometimes a variable $z$ for components),
  as for example in~\cite{gn09a}; and we may think of $y$ as giving a weight for edges.  
  Also for example, for a fixed $1<\mu<3$, we may learn about the random planar graph with $n$ vertices
  and with $\sim \mu n$ edges
  by choosing a suitable value for the edge-weight, see~\cite{gn09a}.
  Further, models in physics involving lattices or more general graphs may attach weights to vertices or edges;
  for example the hard-core model, which is a model for a gas with particles of non-negligible size,
  and which also appears in models for communications networks, see for example~\cite{bs94}.

  We find that many results extend naturally from the uniform case $\tau=(1,1)$ to 
  general $\tau$-weighted random graphs. As well as generalising previous work on the uniform case in this way,  
  we give new results on the 2-core of $R_n$,
  arising from a more combinatorial proof of a key `smoothness' result,
  see~\cite{bcr08,cmcd09}.
  The 2-\!{\em core} of a graph $G$ (sometimes called just the \emph{core}), denoted here by $\core(G)$,
  is the unique maximal subgraph with minimum degree is at least $2$.
  Thus $\core(G)$ is empty if and only if $G$ is a forest; and
  the 2-core may be obtained by repeatedly trimming off leaves.

  To investigate the 
  random graph $R_n \in_{\tau} \ca$ we need to consider how $\tau(\ca_n)$ grows with $n$.
  As in the uniform case, we say that the weighted graph class $\ca,\tau$ has {\em growth constant}
  $\gamma = \gamma(\ca,\tau)$ if $0 \leq \gamma<\infty$ and
  $\tau(\ca_n) = (\gamma +o(1))^n n!$ as $n \to \infty$.
  (Sometimes we insist that $\gamma>0$.)
  Also we say that $\ca,\tau$ is {\em smooth} (or {\em smoothly growing}) if
  $\tau(\ca_{n})/n \tau(\ca_{n-1})$ tends to a limit $\gamma'$ with $0<\gamma'<\infty$ as $n \to \infty$.
  (This is also referred to as the `ratio test' property RT1, see for example~\cite{bb03,{burris2001}}.) 
  It is easy to see that in this case the limit $\gamma'$ must be the growth constant $\gamma$.

  We need some definitions concerning a graph class $\ca$.
  We say that $\ca$ is {\em proper} when it is not the class of all graphs
  and it contains a graph with at least one edge;
  that $\ca$ is {\em decomposable}
  when a graph is in $\ca$ if and only if each component is; that $\ca$ is
  {\em bridge-addable} when, for each graph in $\ca$ and each pair $u$ and $v$ of
  vertices in different components, the graph obtained by adding an edge
  joining $u$ and $v$ must also be in $\ca$; and $\ca$ is
  {\em addable} when it is both decomposable and bridge-addable.
  Also, we say that $\ca$ is {\em minor-closed} if whenever $G \in \ca$ and
  $H$ is a minor of $G$  then $H \in\ca$.

  Let $\ca$ be a proper minor-closed class of graphs.  The minor-minimal graphs not in $\ca$ are the
  \emph{excluded minors}.  From the Robertson and Seymour theory of graph minors~\cite{rs04}, 
  see for example Diestel~\cite{diestel05}, the set of excluded minors must be finite.
  The properties of being decomposable and being addable correspond to simple properties of the excluded minors:
  indeed, $\ca$ is decomposable if and only if each excluded minor is connected,
  and $\ca$ is addable if and only if each excluded minor is 2-connected.
  
  It was conjectured in~\cite{bnw09} 
  that, in the uniform case, every proper minor-closed class of graphs has a growth constant. 
  It is natural to conjecture that this in fact holds for each weighting $\tau$.
  Indeed it is natural to conjecture that we even have smoothness, whenever the growth constant is $>0$.


\subsection{Overview of main results} \label{subsec.overview}

  Our results involve a `well-behaved' weighted class of graphs.
  We shall say later precisely what this means,
  after we have introduced various preliminary definitions.
  We require that such a class is proper, minor-closed and bridge-addable, and satisfies certain further conditions:
  the full definition is given in section~\ref{subsec.wellbdef} below.
  The important thing to note here is that the following classes of graphs are all well-behaved, with any weighting $\tau$:
  any proper, minor-closed, addable class (for example the class of forests, or series-parallel graphs or planar graphs);
  the class $\gs$ of graphs embeddable on any given surface $S$; and
  the class of all graphs which contain at most $k$ vertex-disjoint cycles, for some fixed $k$.

  We shall prove various results about any well-behaved weighted class of graphs $\ca,\tau$  
  and about the corresponding 
  random graph $R_n \in_{\tau} \ca$.
  We sketch some of these results now, in the order in which the proofs run:
  full details appear in Section~\ref{sec.results}.

  First we show the key counting result that $\ca,\tau$ is smooth,
  with some growth constant $\gamma >0$, which is independent of $\nu$. 
  In the process of doing this we learn about the 2-core of $R_n$: 
  in particular we find that, as $n \to \infty$
\begin{equation} \label{eqn.2core}
  v(\core(R_n))/ n \to (1- t) \mbox{ in probability}
\end{equation}
  where $t$ is the unique root with $0<t<1$ to 
  $e^t/t=\lambda \gamma$.
  (Recall that $\tau=(\lambda,\nu)$.)  
  When $\ca$ is the class $\gs$ of graphs embeddable on a given surface $S$
  and $\lambda=1$ 
  then from~\cite{gn09a,cmcd08} we have $\gamma \approx 27.22687$, and so
  $v(\core(R_n)) \approx 0.96184 n$ \whp.
  (We say that a sequence $(A_n)$ of events holds \emph{with high probability} (\whp ) if $\Pr(A_n) \to 1$ as $n \to \infty$.)

  After that, using smoothness, we learn about the `fragments' of $R_n$
  not in the giant component, 
  and in particular we find that
\begin{equation} \label{eqn.conna}
   \pr(R_n \mbox{ is connected}) \; \to \; e^{-D(\rho,\tau)}= e^{- \nu D(\rho,\tau_1)} \mbox{ as } n \to \infty
\end{equation}
  where 
  $\tau_1=(\lambda,1)$ and  $D(\rho,\tau)$ is as described in the next paragraph.
  When $\ca$ is the class of forests, the growth constant $\gamma$ is  $e \lambda_0$
  and the limiting probability of connectedness is $e^{-\frac{\nu}{2\lambda_0}}$.
  When $\ca$ is the class $\gs$ of graphs embeddable on a given surface $S$ and $\lambda=1$
  the limiting probability of connectedness is $e^{-\nu D(\rho,{\bf 1})}  \approx 0.96325^{\nu}$
  (see Section~\ref{subsec.2core}).

  We say that a graph $H$ is {\em freely addable} to a graph class $\ca$
  if the disjoint union $G \cup H$ of $G$ and $H$ is in $\ca$ for each graph $G \in \ca$.
  Let $\cd$ denote the class of connected graphs which are freely addable to $\ca$.
  Observe that if $\ca$ is decomposable then
  each graph in $\ca$ is freely addable to $\ca$, so that $\cd$ is the class of
  connected graphs in $\ca$; and if $\ca=\gs$ then $\cd$ is the class of connected planar graphs.
  Now $D(\rho,\tau)$ is the evaluation of the exponential generating function $D(x,\tau)$ for $\cd$
  (see Subsection~\ref{subsec.boltzmann}), at the radius of convergence $\rho$ for $\ca,\tau$.


\section{Statement of main results}
\label{sec.results}

  The first subsection describes the Boltzmann Poisson random graph corresponding to a decomposable class.
  Then we consider a well-behaved weighted class $\ca$, $\tau$ of graphs and $R_n \in_{\tau} \ca$.
  We describe how the `fragments' not in the giant component of $R_n$ converge in distribution to $R$,
  for the corresponding Boltzmann Poisson random graph $R$, which gives as a corollary the
  result~(\ref{eqn.conna}) on the limiting probability of $R_n$ being connected.
  The next subsection concerns smoothness and the $\core$, and in particular includes the result~(\ref{eqn.2core});
  and then we discuss appearances of subgraphs.
  In the final subsections we say precisely what it means for a graph class to be `well-behaved',
  and then give a sketch plan of the rest of the paper.


\subsection{Boltzmann Poisson random graph}
\label{subsec.boltzmann}

  We introduce a general distribution on the unlabelled graphs corresponding to a weighted class of labelled
  graphs.

  Let $\ca$ be a class of graphs, and let $\cc$ be the class of connected graphs in $\ca$.
  (By convention the empty graph $\emptyset$ is in $\ca$ and not in $\cc$.)
  We define the generating function $A(x,y,z)$ by
\[
  A(x,y,z) = \sum_{n \geq 0} \sum_{G \in \ca_n} \frac{x^n}{n!} y^{e(G)} z^{\kappa(G)}
\]
  and similarly
\[
  C(x,y,z) = z \ \sum_{n \geq 1} \sum_{G \in \cc_n} \frac{x^n}{n!} y^{e(G)}.
\]
  The standard `exponential formula' in a general form (see for example~\cite{fs09,stanley99}) 
  is that when $\ca$ is decomposable we have
\begin{equation} \label{eqn.exp}
  A(x,y,z) = e^{C(x,y,z)}.
\end{equation}

  Let us say that $\ca$ {\em contains components} (or is {\em down-decomposable})
  if each component of each graph in $\ca$ is also in $\ca$.
  Suppose that $\ca$ contains components, and consider any fixed positive $y$ and $z$.
  Then as in~(\ref{eqn.exp}), $C(x,y,z) \leq A(x,y,z) \leq e^{C(x,y,z)}$ for each $x \geq 0$,
  and so 
  the generating functions $A(x,y,z)$ and $C(x,y,z)$ (as functions of $x$) have the same radius of convergence.
  Thus in particular the radius of convergence of $A$ does not depend on~$z$.
  Also we may see that $\ca, \tau$ has growth constant $\gamma$ if and only if $\cc, \tau$ does.
  \smallskip


  For any graph class $\ca$, we let $\cu \ca$ denote the corresponding unlabelled graph class,
  with members the equivalence classes of graphs in $\ca$ under isomorphism.
  Now let $\ca$ be any decomposable class of (labelled) graphs.
  %
  As we shall observe later, 
  we may write its generating function $A(x,y,z)$ in terms of $\cu \ca$
  as
\begin{equation} \label{eqn.Agen}
  A(x,y,z) = \sum_{H \in \cu \ca} \frac{x^{v(H)}y^{e(H)} z^{\kappa(H)}}{\aut(H)}.
\end{equation}
  Here $\aut(G)$ denotes the number of automorphisms of $G$.
  Suppose that we are given 
  $\tau=(\lambda,\nu)$.
  We shall set $y=\lambda$ 
  and $z=\nu$, and write either $A(x,\lambda,\nu)$ or $A(x,\tau)$.
  If we choose $\rho>0$ such that $A(\rho,\tau)$ is finite,
  then we may obtain a natural `Boltzmann Poisson distribution' on $\cu \ca$ --
  see equation~(\ref{eqn.bp-def}) below.
  The uniform case $\tau=(1,1)$ was considered in~\cite{cmcd09}.
  We denote the radius of convergence of $A(x, \tau)$ 
  (as a function of $x$) by $\rho(\ca,\tau)$.

  We need more notation (following~\cite{cmcd09}) to record some of the properties of this distribution.
  For a connected graph $H$ let
  $\kappa(G,H)$ denote the number of components of $G$
  isomorphic to $H$; and for a class $\cal D$ of 
  connected graphs let
  $\kappa(G,{\cal D})$ denote $\sum_{H \in {\cal U \cal D}} \kappa(G,H)$,
  the number of components of $G$ isomorphic to some graph in $\cal D$.
  The notation $X \sim \Po(\mu)$ means that the random variable $X$ has 
  the Poisson distribution with mean $\mu$.
  Recall that a sum of independent Poisson random variables $\Po(\mu_i)$
  has distribution $\Po(\sum_i \mu_i)$, as long as $\sum_i \mu_i < \infty$.
  \begin{theorem} \label{thm.U}
  Consider the weighted graph class $\ca,\tau$ where $\ca$ is decomposable.
  Let $\rho >0$ be such that $A(\rho,\tau)$ is finite, and let
\begin{equation} \label{eqn.mudef}
  \mu(H) = \frac{\rho^{v(H)} \lambda^{e(H)} \nu^{\kappa(H)}}{\aut(H)} \; \mbox{ for each } H \in \cu\ca
\end{equation}
  (so that $A(\rho,\tau)= \sum_{H \in \cu\ca} \mu(H)$ by equation~(\ref{eqn.Agen})).
  Let the `Boltzmann Poisson random graph' $R=R(\ca,\rho,\tau)$ take values in $\cu \ca$, with
  \begin{equation} \label{eqn.bp-def}
  \pr[R=H] = \frac{\mu(H)}{A(\rho,\tau)}  \;\; \mbox{ for each } H \in \cu \ca.
  \end{equation}
  Also, let $\cc$ denote the class of connected graphs in $\ca$.
  
  Then the random variables $\kappa(R,H)$ for $H \in {\cal U}{\cal C}$ are independent,
  with $\kappa(R,H) \sim \Po(\mu(H))$.
  \end{theorem}
  In particular, since $C(\rho,\tau) =  \sum_{H \in \cu\cc} \mu(H)$
  (by equation~(\ref{eqn.Agen}) applied to $\cc$) we have $\kappa(R) \sim \Po(C(\rho,\tau))$.
  %
%


\subsection{Fragments and connectivity}
\label{subsec.connresults}

   The {\em big component} $\Bigc(G)$ of a graph $G$ is the
  (lexicographically first) component with the most vertices, and
  $\Frag(G)$ is the {\em fragments} subgraph induced on the vertices
  not in the big component (thus  $\Frag(G)$ may be empty).
  Denote the numbers of vertices in $\Bigc(G)$ and
  $\Frag(G)$ by $\bigc(G)$ and $\frag(G)$ respectively, so
  $\bigc(G)+\frag(G)= v(G)$. 
  %
  %
  We consider $R_n \in_{\tau} \ca$, and focus on the limiting distribution of the random graph $\Frag(R_n)$.
  It is convenient to deal with the random unlabelled graph $F_n$ 
  corresponding to $\Frag(R_n)$.
  We use $\to_{TV}$ to denote convergence in total variation (or in distribution). 
  %

\begin{theorem} \label{thm.Frag}
  Let the weighted graph class $\ca, \tau$ be well-behaved,
  and let $\rho = \rho(\ca,\tau)$.
  Let $\cf_{\ca}$ be the class of graphs freely addable to $\ca$, with exponential generating function $F_{\ca}$.
  Then $0<\rho<\infty$ and $F_{\ca}(\rho,\tau)$ 
  is finite; and for the 
  random graph $R_n \in_{\tau} \ca$, the random unlabelled graph
  $F_n$ corresponding to $\Frag(R_n)$ satisfies $F_n \to_{TV} R$,
  where $R$ is the Boltzmann Poisson random graph for $\cf_{\ca},\rho,\tau$
  defined in~(\ref{eqn.bp-def}).
  Further, $\E[v(R)] = \rho \, D'(\rho,\tau) < \infty$,
  where $\cd$ is the class of connected graphs in $\cf_{\ca}$. 
\end{theorem}

\begin{corollary} \label{cor.comps}

  (a)
  For any given distinct graphs $H_1,\ldots,H_k$ in $\cu \cd$ the $k$ random
  variables $\kappa(F_n,H_i)$ are asymptotically independent with
  distribution $\Po(\mu(H_i))$.

  (b) For any class $\tilde{\cd} \subseteq \cd$ 
  we have $\kappa(F_n,\tilde{\cd}) \to_{TV} \Po(\tilde{D}(\rho,\tau))$,
  and each moment of $\kappa(F_n,\tilde{\cd})$ tends to that of $\Po(\tilde{D}(\rho,\tau))$.

  (c)  As a special case of part (b),
  $\kappa(F_n) \to_{TV} 1+ \Po(D(\rho,\tau))$, and as $n \to \infty$
  we have
  $\pr[F_n \mbox{ is connected }] \to e^{-D(\rho,\tau)} = F_{\ca}(\rho,\tau)^{-1}$,
  $\E[\kappa(F_n)] \to 1 + D(\rho,\tau)$, and the variance of
  $\kappa(F_n)$ tends to $D(\rho,\tau)$.

  (d)
  The random number of vertices $v(F_n)=\frag(R_n)$ satisfies
  $v(F_n) \to_{TV} v(R)$; that is, for each non-negative integer $k$
  \[ \pr[v(F_n) = k] \to \frac{1}{F_{\ca}(\rho,\tau)}  \frac{ \rho^{k}}{k!} \tau((\cf_{\ca})_k) \; \mbox{ as } n \to \infty, \]
  and similarly $e(F_n) \to_{TV} e(R)$.
\end{corollary}

 
  \noindent
  In the uniform case $\tau=(1,1)$ of the above result, part (c) on $\kappa(R_n)$ and part (d) on $v(F_n)$
  extend for example Theorems 5.2 and 5.3 of~\cite{cfgmn11}.
\medskip

\noindent
\emph{Trees and forests}
  Let us illustrate the above results for the classes $\ct$ of trees and $\cf$ of forests.
  Denote $\cf$ by $\ca$ temporarily (just the next three times). 
  The class $\ca$ is minor-closed and addable, and so it is well-behaved; and the class $\cf_{\ca}$
  of graphs freely addable to $\ca$ is just $\cf$ again.
  
  By Cayley's formula $\tau(\ct_n)= n^{n-2} \lambda^{n-1} \nu$, and by Stirling's formula
  $(n!)^{1/n} \sim n/e$.  Thus $(\tau(\ct_n)/n!)^{1/n} \to e \lambda$ as $n \to \infty$,
  so that $\ct,\tau$ has growth constant $e \lambda$.
  By Lemma~\ref{lem.tauconn} below
\[ \tau(\cf,\tau) \geq \tau(\ct,\tau) \geq e^{-\nu/\lambda} \tau(\cf,\tau)\]
  and it follows that $\cf,\tau$ also has growth constant $e \lambda$
  (see also Section~\ref{subsec.tau-vary} below).
  Thus $\rho=\rho(\cf,\tau)= 
  (e \lambda)^{-1}$.
  Recall that
\begin{equation} \label{eqn.treeformulae}
  \sum_{n \geq 1} \frac{n^{n-2}}{e^n n!} = \frac12 \;\; \mbox{ and } \;\; \sum_{n \geq 1} \frac{n^{n-1}}{e^n n!} = 1.
\end{equation}
  (One way to see these results is to consider the exponential generating functions $U(z)$ for (Cayley) trees and
  $T(z)$ for rooted trees respectively, where
  $U(z) = \sum_{n \geq 1} n^{n-2} z^n/n!$ and $T(z) = \sum_{n \geq 1} n^{n-1} z^n/n!$.
  Since $T(z)=ze^{T(z)}$ we find $T(1/e)=1$, and since $U(z)=T(z)-T^2(z)/2$ we find $U(1/e)=1/2$,
  see for example Stanley~\cite{stanley99} chapter 5, or Flajolet and Sedgewick~\cite{fs09} section II.5.).
  
  Thus the exponential generating function $T$ (for the weighted case) satisfies
\[
  T(\rho, \tau)= \nu \sum_{n \geq 1} n^{n-2}
  \lambda^{n-1}(e \lambda)^{-n}/n! =  \frac{\nu}{\lambda} \sum_{n
  \geq 1} \frac{n^{n-2}}{e^n n!} = \frac{\nu}{2 \lambda}.
\]
  and
\[
  T'(\rho, \tau)
  = \nu \sum_{n \geq 1} n^{n-1}\lambda^{n-1}(e \lambda)^{-(n-1)}/n!
  = e \nu \sum_{n \geq 1} \frac{n^{n-2}}{e^n n!}
  = e \nu.
\]
  Now consider $R_n \in_{\tau} \cf$.
  It follows from Corollary~\ref{cor.comps} part (c) that,
  as $n \to \infty$, $\kappa(R_n)$ converges in distribution to $1+\Po(\nu/2\lambda)$, and so in particular
\begin{equation} \label{eqn.for1}
  \pr(R_n \mbox{ is connected }) = \frac{\tau(\ct_n)}{\tau(\cf_n)} \to
  e^{-\frac{\nu}{2 \lambda}} 
\end{equation}
  and so
\begin{equation} \label{eqn.for2}
  \tau(\cf_n) \sim \nu e^{\frac{\nu}{2 \lambda}} \ n^{n-2} \lambda_0^{n-1} .
\end{equation}
  %
  Also, by part (d), as $n \to \infty$, $\Frag(R_n)$ converges in distribution to $R$,
  so $\frag(R_n)$ converges in distribution to $v(R)$, and indeed in this case we may see
  that also $\E[\frag(R_n)] \to \E[v(R)]$ as $n \to \infty$ (this follows using the formulae above for
  $\tau(\ct_n)$ and $\tau(\cf_n)$, and arguing as in the proof of Proposition 5.2 of~\cite{cmcd08}),
  where by Theorem~\ref{thm.Frag} $\E[v(R)] = \rho T'(\rho,\tau) = \frac{\nu}{\lambda}$. 
  \medskip

  Now consider (vertex-) rooted trees, which also have growth constant $e \lambda_0$.
  The exponential generating function $T^o$ satisfies
\[
  T^o((e \lambda)^{-1}, \tau)
  = \nu \sum_{n \geq 1} n^{n-1} \lambda^{n-1}(e \lambda_0)^{-n}/n!
  =  \frac{\nu}{\lambda} \sum_{n \geq 1} \frac{n^{n-1}}{e^n n!}
  = \frac{\nu}{\lambda},
\]
  where we have used~(\ref{eqn.treeformulae}).
  This result is used in the introduction to Section~\ref{subsec.conn-smooth}.


\subsection{Smoothness and $\core(R_n)$} \label{subsec.2core}

  Our next theorem says that a well-behaved weighted graph class $\ca,\tau$ is smooth,
  and gives results on $\core(R_n)$ for the 
  random graph $R_n \in_{\tau} \ca$.
  %
  Recall that $\cf,\tau$ has growth constant $\lambda e$, where $\cf$ is the class of forests.

\begin{theorem} \label{thm.sumup}
   Let the weighted graph class $\ca,\tau$ be well-behaved, with growth constant $\gamma$;
   let $\cc$ denote the class of connected graphs in $\ca$;
   and let $R_n \in_{\tau} \ca$.
   Then
\begin{description}
 \item{(a)} Both $\ca, \tau$ and $\cc,\tau$ 
   are smooth with growth constant $\gamma$, and $\gamma \geq \lambda e$.
   
 \item{(b)} Let $\cb$ denote the class $\cc^{\delta\geq 2}$ of graphs in $\cc$ with minimum degree at least~2.
   If $\gamma > \lambda e$ then 
   $\cb,\tau$ has growth constant $\beta$ where $\beta$ is the unique root
   $> \lambda$ to $\beta e^{\lambda/\beta} = \gamma$;
   and if $\gamma \leq \lambda e$ then $\rho(\cb,\tau) \geq \lambda^{-1}$.
  %
   \item{(c)}
   If $\gamma > \lambda e$ let $\alpha= 1-x$ where $x$ is the unique root $<1$ to $xe^{-x}= \lambda/\gamma$;
   and otherwise let $\alpha=0$.
   Then 
   for each $\eps>0$
 \begin{equation} \label{eqn.expcorebd}
   \pr(|v(\core(R_n))-\alpha n| > \eps n) = e^{-\Omega(n)}.
 \end{equation}
   %

   \item{(d)}
   Let $\cal T$ denote the class of trees and let $\cd$ denote the class of connected
   graphs which are freely addable to $\ca$,
   with generating functions $T$ and $D$ 
   respectively; and let $\rho=1/\gamma$.  Suppose that $\gamma > \lambda e$
   (so the probability that $\core(R_n)$ is empty is $e^{-\Omega(n)}$ by part (c)). Then
   $T(\rho,\tau) < D(\rho,\tau) < \infty$, and
   the probability that $\core(R_n)$ is non-empty and connected tends to $e^{T(\rho,\tau)-D(\rho,\tau)}$ as $n \to \infty$.

\end{description}
\end{theorem}

\smallskip

\noindent
\emph{Graphs on surfaces}
  Let us illustrate the theorem above for the class $\cg^S$
  of graphs embeddable on a given surface $S$.
  It was shown in~\cite{cmcd08} that for any fixed surface $S$, the class $\cg^S$
  of graphs embeddable on $S$ has growth constant~$\gamma$, where
  $\gamma$ is the {\em planar graph growth constant} (the same $\gamma$ for each surface),
  and recently this was very much improved to give an asymptotic formula for $|\gs_n|$,
  see~\cite{cfgmn11,bg09}.
  From Gim\'enez and Noy~\cite{gn09a} we have $\gamma \approx 27.226878$.
  For any weighting $\tau$, 
  the weighted class $\cg^S,\tau$ is well-behaved
  (this is part of lemma~\ref{lem.wellb} below),
  and so by Theorem~\ref{thm.sumup} we see that $\cg^S, \tau$ is smooth;
  and when we specialise to the uniform case 
  we obtain the result of~\cite{bcr08} that the (uniform) class $\cg^S$ is smooth
  (this also follows directly from the recent asymptotic formula for $|\cg^S_n|$ mentioned above).
  Further we obtain new information on the core of a uniform random graph $R_n \inu \gs$, as follows.
  
  Solving $\beta e^{1/\beta}= \gamma$
  (using the more accurate figure for $\gamma$ 
  in Theorem 1 in~\cite{gn09a}) gives $\beta = \beta_0 \approx 26.207554$,
  and solving $\alpha = 1-1/\beta$ gives $\alpha = \alpha_0 \approx 0.961843$.
  Thus the class $\cb$ of (connected) 
  graphs in $\gs$ with minimum degree at least 2 
  has growth constant $\beta_0$ and 
  $v(\core(R_n)) \approx \alpha_0 n$ \mbox{\whp.}
  The growth constant $\beta_0$ 
  is only slightly larger than the growth constant $\approx 26.18412$ for 2-connected graphs in
  $\gs$, from ~\cite{bgw02,gn09a}. Also the class $\cd$ of connected freely addable graphs
  is the class of all connected planar graphs,
  and from Corollary 1 in~\cite{gn09a} we have $e^{-D(\rho)} \approx 0.963253$,
  where $\rho=1/\gamma$.
  Further 
  $e^{T(\rho)} \approx 1.038138$,
  so by Theorem~\ref{thm.sumup} part (d) the probability that $\core(R_n)$ is connected $\approx 0.999990$
  (for large $n$).
  Thus the probability that $\core(R_n)$ is not connected $\approx 10^{-5}$.
  For comparison note that $\pr(\Frag(R_n) = C_3) \sim e^{-D(\rho)} \rho^3/6 \approx 8 \cdot 10^{-6}$.




\subsection{Appearances theorem}
\label{subsec.apps}

  It is often useful to know that for $R_n \in_{\tau} \ca$,
 \whp $R_n$ contains many disjoint copies of a given connected graph $H$.
  Suppose that $H$ has a specified root vertex $r$.  We say that $H$ is
  {\em freely attachable} 
  to $\ca$ if, given any graph $G \in \ca$ 
  and vertex $v \in G$,
  the graph formed from the disjoint union $G \cup H$ 
  by adding the edge $vr$ is in $\ca$.

  Let $H$ be a graph on the vertex set $[h]=\{1,\ldots,h\}$, and let $G$
  be a graph on the vertex set $[n]$ where $n>h$.
  Let $W \subset V(G)$ with $|W|=h$, and let the \emph{root} $r_W$ be
  the least element in $W$.  We say that $H$ has a
  {\em pendant appearance} at $W$ in $G$ if
  (a) the increasing bijection from $[h]$ to $W$ gives an isomorphism
  between $H$ and the induced subgraph $G[W]$ of $G$; and
  (b) there is exactly one edge in $G$ between $W$ and the rest of $G$,
  and this edge is incident with the root $r_W$.
  We let $f_H(G)$ be the number of pendant appearances of $H$ in $G$,
  that is the number of sets $W \subseteq V(G)$ such that $H$ has a pendant appearance at $W$ in $G$.
  The next theorem extends results in~\cite{msw05,msw06,cmcd08,cmcd09}. 

\begin{theorem} \label{thm.appearances1}
  Let the weighted graph class $\ca, \tau$
  have growth constant~$\gamma$,
  and let $R_n \in_{\tau} \ca$.
  Let the connected graph $H$ 
  be freely attachable to $\ca$.
  Then there exists $\alpha>0$ such that
\[
  \Pr[f_H(R_n) \leq \alpha n] = e^{-\Omega(n)}.
\]
\end{theorem}

  This result shows for example that, if the $k$-leaf star rooted at its centre is freely attachable to
  $\ca$, then $R_n$ has linearly many vertices of degree $k+1$, with exponentially small failure
  probability.
  It is possible to extend the theorem to consider graphs $H$ with (slowly) growing size, see for example Theorem 3.1 in~\cite{cb08},
  but we do not pursue that here.
  If $\ca,\tau$ is well-behaved then we can be more precise about $f_H(R_n)$,
  extending Proposition 1.9 of~\cite{cmcd09}. 

\begin{prop}  \label{thm.apps2}
  Let the weighted graph class $\ca, \tau$ be well-behaved with growth constant $\gamma$,
  and let $R_n \in_{\tau} \ca$.
  Let the connected graph $H$ 
  be freely attachable to $\ca$.
  Then 
\[ \frac{f_H(R_n)}{n} \to \; \lambda \cdot \frac{\lambda^{e(H)}}{\gamma^{v(H)} v(H)!}
  \;\;\mbox{ in  probability as } n \to \infty.\]
\end{prop}
  The same result holds if we count disjoint pendant appearances.
  Indeed, if $\tilde{f}_H(R_n)$  
  denotes the number of pendant appearances of $H$ in $R_n$
  that share a vertex or the root edge with some other pendant appearance of $H$,
  then $\E[\tilde{f}_H(R_n)] = O(1)$.



\subsection{Definition of a `well-behaved' graph class}
\label{subsec.wellb-def}

  Now at last in this section we can say precisely what we mean by a well-behaved class.
  We need first to introduce the notions of a `dichotomous' class of
  graphs, and of a graph class `maintaining at least factorial growth'.
  Along the way we introduce `very well-behaved' graph classes.

\subsubsection{Dichotomous classes of graphs}
\label{subsec.dichotomous}

  Recall that, given a class $\ca$ of graphs, the graph $H \in \ca$ is freely addable to $\ca$
  if the disjoint union $G \cup H \in \ca$ whenever $G \in \ca$.
  We denote the class of graphs which are freely addable to $\ca$ by $\cf_{\ca}$. 
  For example, if $\ca$ is the class $\gs$ of graphs embeddable on a fixed surface $S$
  then $\cf_{\ca}$ is the class $\cp$ of planar graphs.
  %
  Observe that
  if $\ca$ is minor-closed and bridge-addable then $\cf_{\ca}$ is minor-closed and addable.

  We say that a graph $H \in \ca$ is {\em limited} in $\ca$ if $kH$ is not in $\ca$ for some positive integer $k$.
  Here $k H$ denotes the disjoint union of $k$ copies of $H$.
  For example, if $\ca$ is $\gs$ as above then the graphs in $\ca$ which are limited in $\ca$ are the non-planar graphs,
  which are exactly the non-freely addable graphs.  If $\ca$ is decomposable then no graph in $\ca$ is limited in $\ca$.
  Indeed if $H$ is freely addable to a class $\ca$ then $H$ is not limited in $\ca$
  (recall that each graph in na decomposable class is freely addable to the class).

  We are interested in classes $\ca$ of graphs such that each graph $G \in \ca$ is either freely-addable or limited
  (as we have noted it cannot be both): let us call such a graph class ({\em freely-addable/limited}) {\em dichotomous}. 
  From what we have just seen, for any surface $S$ the class $\gs$ is dichotomous.
  If $\ca$ is decomposable then $\cf_{\ca}=\ca$ so $\ca$ is dichotomous.
  If $\ca$ is the class $\ex(k C_3)$ of graphs with at most $k$ vertex-disjoint
  cycles, then $\cf_{\ca}$ is the class of forests and each graph
  with a cycle is limited, so $\ca$ is dichotomous.
  %
  An example of a non-dichotomous class is the class $\ex(C_3 \cup C_4)$ of graphs with no minor $C_3 \cup C_4$,
  where $C_3$ is neither freely addable nor limited.




\subsubsection{Very well-behaved classes of graphs}
\label{subsubsec.vwellbdef}

  It is conjectured~\cite{bnw09} that any proper minor-closed class $\ca$ of graphs has a growth constant,
  and it is natural to conjecture similarly that $\ca,\tau$ always has a growth constant.
  Part of the definition of $\ca,\tau$ being well-behaved 
  will require that there is a growth constant.

  We (temporarily) call the weighted class $\ca,\tau$ of graphs {\em very well-behaved} if it is
  minor-closed, bridge-addable and dichotomous; and if it either is decomposable, or it is closed under subdividing edges
  and has a growth constant.

  From what we have already seen, 
  to show that $\ca,\tau$ must be very well-behaved when $\ca$ is a 
  proper minor-closed addable class, or a class $\gs$, it suffices to show that the growth constant must exist.
  This is done in Sections~\ref{subsec.addable-gc} and~\ref{subsec.gs-gc}.
  
  Now let $\ca$ be the class of graphs with at most $k$ vertex-disjoint cycles.
  Then $\ca$ is minor-closed, bridge-addable, dichotomous and closed under subdividing edges.
  In the uniform case this class has growth constant $2^ke$~\cite{km08b}, and so it is
  very well-behaved.  Furthermore, straightforward adaptations of the proof in~\cite{km08b} shows that the
  weighted class $\ca,\tau$  has a growth constant, 
  and so it is very well-behaved. 
  However, 
  consider for example the class of graphs with no two vertex-disjoint cycles of length at least 4:
  this class is not decomposable nor closed under subdividing edges, and so it is not well behaved (in the uniform case). 
  To cover such further graph classes, we weaken the condition and define
  a larger class of `well-behaved' graph classes.  Unfortunately the definition is 
  more involved.


\subsubsection{Maintaining at least factorial growth}
\label{subsec.maint}

  We shall want to consider classes of graphs which we can show do not have any sudden dips
  in their growth rate.  Let us say that a weighted graph class $\ca,\tau$ of graphs
  {\em maintains at least factorial growth}
  if there exist an $\eta>0$ and a function $g(n)=(1+o(1))^n$ such that for each $n$ and
  each $j$ with $1 \leq j < n$ we have
\begin{equation} \label{eqn.mafg}
   \tau(\ca_n) \geq  \tau(\ca_{n-j}) \ (n)_j \ \eta^j \ g(n).
\end{equation}
  An equivalent condition avoiding the function $g$ is that there exist an $\eta>0$ such that for each $\eps>0$,
  for each sufficiently large $n$, for each $1 \leq j <n$ 
\begin{equation} \label{eqn.mafg2}
   \tau(\ca_n) \geq  \tau(\ca_{n-j}) \ (n)_j \ \eta^j \ e^{- \eps n}.
\end{equation} 
  It follows easily from lemma~\ref{lem.tauconn} (a) below that,
  if $\ca$ is bridge-addable and $\cc$ is the class of connected graphs in $\ca$,
  then $\ca,\tau$ maintains at least factorial growth if and only if $\cc,\tau$ does.
  
  These equivalent conditions are weaker than having a growth constant.
  Several weighted graph classes may be shown easily to maintain at least factorial growth,
  for example if $\ca,\tau$ has a growth constant or if $\ca$ is closed under subdividing edges 
  -- see Section~\ref{subsec.growth} below. 
  Since $(n)_j \geq (n/e)^j$ for $j=1,\ldots,n$, it would make no difference if we replaced
  $(n)_j \eta^j$ in the definition by $(\eta n)^j$.
  Also, to prove that $\ca,\tau$ maintains at least factorial growth, by Lemma~\ref{lem.malfgnod} it suffices to show that
  there is a $\delta >0$ such that~(\ref{eqn.mafg}) or~(\ref{eqn.mafg2}) holds for each $j$ with $1 \leq j \leq \delta n$.


\subsubsection{Well-behaved classes of graphs}
\label{subsec.wellbdef}

  We may at last say exactly what `well-behaved' means.
  We start with the definition of a very well-behaved graph class $\ca$,
  and we simply replace the condition that $\ca$ be either decomposable or closed under
  subdividing edges by the condition that the subclass $\ca^{\delta \geq 2}$
  is either `as small as the paths' or it is `consistently large'.
  (Recall that $\ca^{\delta \geq 2}$ is the class of graphs in $\ca$ with minimum degree at least~2.)
  It will be easy to check that any very well-behaved graph class is well-behaved.
  \smallskip

  \noindent
  {\bf Definition}
  The weighted class $\ca,\tau$ of graphs is {\em well-behaved} if it is
  minor-closed, bridge-addable, dichotomous, and 
  has a growth constant; and if the weighted class 
  $\ca^{\delta \geq 2},\tau$ either (a) is empty or has radius of convergence at least $\lambda^{-1}$ 
  or 
  (b) maintains at least factorial growth.
  
  \noindent
  From our observations in Section~\ref{subsec.boltzmann}, we obtain an equivalent condition if we replace the
  assumption that $\ca$, $\tau$ has a growth constant by the assumption that $\cc,\tau$  has a growth constant,
  where $\cc$ is the class of connected graphs in~$\ca$.  Similarly, we could replace $\ca^{\delta \geq 2}$ by 
  the class of connected graphs in $\ca$ with minimum degree at least 2.

\begin{lemma} \label{lem.wellb}
  Every very well behaved weighted graph class is well behaved,
  and in particular the weighted class $\ca,\tau$ is well behaved in the following cases, with any $\tau$:
\begin{description}
  \item{(a)}
    $\ca$ is minor-closed and addable, 
  \item{(b)}
    $\ca$ is the class $\gs$ of graphs embeddable on any given surface~$S$, 
  \item{(c)}
    $\ca$ is the class of graphs which contain at most $k$ vertex-disjoint cycles,
    for any given $k$. 
\end{description}
\end{lemma}



\subsection{Plan of the rest of the paper}
\label{subsec.plan}

  The next section collects and proves various preliminary general results,
  and contains a proof of Lemma~\ref{lem.wellb} above which shows that certain graph classes
  are well-behaved.
  After that, in Section~\ref{sec.proofU}, we prove the results stated in Section~\ref{subsec.boltzmann} 
  on the Boltzmann Poisson random graph.  The next section proves most of the results
  presented in Section~\ref{subsec.2core} on smoothness and the $\core$;
  and the smoothness results allow us, after a brief
  section on Poisson convergence, to prove the results on $\Frag(R_n)$ and connectedness
  in Section~\ref{subsec.connresults}.
  After that we prove the results on appearances given in Section~\ref{subsec.apps},
  and finally we make some concluding remarks.



\section{Preliminary general results}
\label{sec.prelims}

  This section presents various preliminary general results,
  and gives a proof of Lemma~\ref{lem.wellb}, which shows that certain interesting graph classes are well-behaved.
  %


\subsection{Connectivity bounds for a bridge-addable class}

\label{sec.comps}

  We start with a lemma taken from~\cite{cmcd12}, which will be used several times in this paper.
  Part (a) is a special case of Theorem~2.1 in~\cite{cmcd12}, and extends Theorem 2.2 of~\cite{msw05}:
  part (b) is a special case of Theorem~2.2 in~\cite{cmcd12}, and extends Lemmas 2.5 and 2.6 of~\cite{cmcd09}.
  (The paper~\cite{cmcd12} also gives asymptotic versions of these results,
  in the case when the class is closed also under deleting bridges, which match the results for forests described
  in~\ref{subsec.connresults}.)

\begin{lemma} \label{lem.tauconn}
  Let the finite non-empty weighted set $\ca,\tau$ of graphs be bridge-addable,
  and let $R \in_{\tau} \ca$.
  Then\\ (a) $\kappa(R)$ is stochastically at most $1+ \Po(\nu/\lambda)$, and so in particular\\
   $\pr(R \mbox{ is connected}) \geq e^{-\nu/\lambda}$;
  and\\ (b) $\E[\frag(R)] < 2  \nu /\lambda$.
\end{lemma}






\subsection{Growth constant for an addable class}
\label{subsec.addable-gc}

  The following result is an extension of 
  Proposition 1.1 in~\cite{cmcd09}, and its proof follows similar lines.

\begin{lemma} \label{lem.addable-gc}
  Let $\ca$ be a non-empty addable subclass of a proper minor-closed class of graphs,
  and consider any weighting $\tau$.
  Then there is a constant $\gamma$ with $0<\gamma<\infty$, which is independent of $\nu$,
  such that $(\tau(\ca_n)/n!)^{1/n} \to \gamma$ as $n \to \infty$.
\end{lemma}

\begin{proof}
  Let $\cc$ be the class of connected graphs in $\ca$.
  Then $\tau(\cc_n) \geq \tau(\ca_n) e^{-\nu/\lambda}$ by Lemma~\ref{lem.tauconn} (a),
  and so for all positive integers $a$ and $b$
\begin{eqnarray*}
  \tau(\ca_{a+b}) & \geq &
  \frac12 {a+b \choose a} \ \tau(\cc_a)\ \tau(\cc_b)\\
  & \geq &
  2  (a+b)! \ \frac{\tau(\ca_a) e^{-\nu/\lambda} }{2\ a!} \ \frac{\tau(\ca_b)
  e^{-\nu/\lambda}}{2\ b!}.
\end{eqnarray*}
  Thus, if we set $f(n)= \frac{\tau(\ca_n)\, e^{-2 \nu/\lambda} }{ 2\, n!}$ then
  $f(a+b) \geq f(a) f(b)$, that is $f$ is supermultiplicative.

  Since $\ca$ is a subclass of a proper minor-closed class of graphs, there is a constant $c_1$
  such that $|\ca_n| \leq n! c_1^n$, see~\cite{nstw06} (or~\cite{dn10} for a different proof).
  Also there is a constant $c_2$ such that each graph in $\ca$ has average degree at most $c_2$,
  by a result of
  Mader~\cite{mader68} (see also for example Diestel~\cite{diestel05}).
  Hence
\[
  \tau(\ca_n) \leq |\ca_n| \max\{1, \lambda^{n-1}\} \max\{ 1, \lambda_1^{c_2 n/2}\}
  \max \{\nu, \nu^n \}. \]
  Thus 
  $\gamma = \sup_n f(n)^{1/n}$ satisfies $0<\gamma<\infty$;
  and since $f$ is supermultiplicative it follows by Fekete's lemma
  (see for example~\cite{lint-wilson} Lemma 11.6) that
  as $n \to \infty$ we have $f(n)^{1/n} \to \gamma$ and so also $(\tau(\ca_n)/n!)^{1/n} \to \gamma$.
  Finally note that $\gamma$ cannot depend on $\nu$ since $\ca$ contains all components
  (see the discussion following~(\ref{eqn.exp})).
\end{proof}



\subsection{Growth constant for $\gs,\tau$}
\label{subsec.gs-gc}

  Since the class $\cp$ of planar graphs is minor-closed and addable,
  we know from the last subsection that, with any weighting $\tau$, the weighted graph class 
  $\cp,\tau$ has a growth constant $\gamma(\cp,\tau)$, which does not depend on $\nu$.
  We may show that $\gs,\tau$ has the same growth constant $\gamma(\cp,\tau)$,
  by induction on the Euler genus of $S$, following the treatment of the uniform case in~\cite{cmcd08}.


\subsection{Freely-addable graphs and dichotomous graph classes}
\label{subsec.limited}

  It was shown in~\cite{cmcd08} that, for $R_n \inu \gs$ (the uniform case), the probability that $\Frag(R_n)$ is
  non-planar is $O(\ln n / n)$: here we improve and extend this result.
  Let us write $G \succeq_m H$ to mean that $G$ has a minor (isomorphic to) $H$.
  First we give a lemma concerning a single unwanted minor.


%


\begin{lemma} \label{lem.Frag2}
  Let $\ca$ be bridge-addable, let $R_n \in_{\tau} \ca$, let $H$ be a connected graph,
  and let $k$ be a non-negative integer.  Then
\begin{equation}  \label{eqn.bound0}
  \pr( \Frag(R_n) \succeq_m H) \leq \frac{3 \nu k}{2 \lambda} \frac{\ln n}{n} + \frac{6 \nu}{\lambda n} + \pr(R_n \succeq_m (k+1)H).
\end{equation}
  Further if $H$ is 2-connected then we may improve the first term in the bound to
  $\frac{3 \nu k}{2 \lambda v(H)} \frac1{n}$.
\end{lemma}

\medskip

\begin{proof} 
  Let $\omega=\omega(n)= \lfloor n/3 \rfloor$.
  Let $\cb^{j}_n$ be the class of graphs $G \in \ca_n$ such that $G \in \ex(k+1)H$,
  $\frag(G) \leq \omega$, $\Frag(G) \succeq_m H$, and there are exactly $j$ vertices in the
  lex-first component $C$ of $\Frag(G)$ with minor $H$.
  Given a graph $G \in \cb^{j}_n$, add any edge between this component $C$ and a vertex in $\Bigc(G)$, to form $G'$.
  This gives $j \cdot \bigc(G) \geq j \cdot 2n/3$ constructions of graphs $G' \in \ca$.

  Each graph $G'$ constructed can have at most $k$ (oriented) bridges $uv$ such that the component of $G'-uv$
  containing $u$ has order at least $2n/3$, and the component containing $v$ has order $j$ and has a minor $H$
  (since otherwise the original graph $G$ is not in $\ex(k+1)H$).
  Thus $G'$ can be constructed at most $k$ times.  Hence
\[
  j \frac{2n}{3} \ \frac{\lambda}{\nu} \tau(\cb^j_n) \leq k \tau(\ca_n)
\]
  so
\[
  \tau(\cb^j_n) \leq \frac{1}{j} \frac{3 \nu k}{2 \lambda n} \tau(\ca_n).
\]
  Let $\cb_n= \cup_{j \leq \omega} \cb^j_n$.  Since $\sum_{v(H) \leq j \leq \omega} 1/j \leq \ln \omega$ we have
\[
  \tau(\cb_n) \leq \frac{3 \nu k}{2 \lambda} \frac{\ln n}{n} \tau(\ca_n).
\]
  Thus
\begin{eqnarray}
  && \tau(\{G \in \ca_n, \Frag(G) \succeq_m H\})\\ \nonumber
  & \leq &
   \frac{3 \nu k}{2 \lambda} \frac{\ln n}{n} \tau(\ca_n)
   + \tau(\{G\!\in\!\ca_n\!: \frag(G)\!>\!\omega\}) + 
   \tau(\{G\!\in\!\ca_n\!: G \succeq_m (k\!+\!1)H \}) \label{eqn.bound1}
\end{eqnarray}
  But recall from Lemma~\ref{lem.tauconn} (b) that $\E[\frag(R_n)] < 2 \nu/\lambda$,
  and so $\pr(\frag(R_n) > n/3) < \frac{6 \nu}{\lambda n}$.
  Hence we may divide by $\tau(\ca_n)$ in~(\ref{eqn.bound1}) to obtain~(\ref{eqn.bound0}).
  \smallskip
 
  Now suppose that $H$ is 2-connected. 
  Let $G \in {\cb}_n$, where $\cb$ is as above.  Then some block of some component
  of $\Frag(G)$ has a minor $H$.
  Add any edge between a vertex in such a block and a vertex in $\Bigc(G)$, to form $G'$.
  This gives at least $v(H) \cdot {\bigc(G)} \geq 2 v(H) n/3$ constructions of graphs $G' \in \ca$.

  Each graph $G'$ constructed can have at most $k$ (oriented) bridges $uv$ such that the component of $G'-uv$
  containing $u$ has order at least $2n/3$, and the component containing $v$ is in a block with a minor $H$.
  Thus $G'$ can be constructed at most $k$ times. Hence
\[
   v(H) \frac{2n}{3} \ \frac{\lambda}{\nu} \tau(\tilde{\cb}_n) \leq k \tau(\ca_n)
\]
  and we may proceed as before.
\end{proof}
  \medskip

  \noindent
  Recall that $\cf_{\ca}$ denotes the class of graphs which are freely addable to~$\ca$.  The last
  lemma yields:

\begin{lemma} \label{lem.Frag}
  Let $\ca$ be minor-closed, bridge-addable and dichotomous; and let $R_n \in_{\tau} \ca$ for $n=1,2,\ldots$.
  Then there is a constant $c$ such that 
\[
  \pr(\Frag(R_n) \not\in \cf_{\ca}) \leq \frac{c \nu}{n \lambda} \;\;\; \mbox{ for each } \;n.
\]
\end{lemma}

\begin{proof} 
  As we noted in Section~\ref{subsec.dichotomous}, $\cf_{\ca}$ is minor-closed and addable,
  with a finite set of excluded minors each of which is 2-connected.
  Thus it suffices to prove that for each excluded minor $K$ for $\cf_{\ca}$ there is a constant $c_K$ such that
  $\pr(\Frag(R_n) \succeq_m K) \leq \frac{c_K \nu}{n \lambda}$ for each $n$.
  But if $K \in \ca$ then $K$ must be limited in $\ca$ since $\ca$ is dichotomous,
  so there is a $k$ such that $(k+1) K \not\in \ca$;
  and now we may use Lemma~\ref{lem.Frag2} for the 2-connected graph~$K$.
  \end{proof}
  


\subsection{Maintaining at least factorial growth}
\label{subsec.growth}

  We shall want to consider weighted graph classes $\ca,\tau$ (with each vertex degree at least 2) which
  maintain at least factorial growth, as defined in Section~\ref{subsec.maint}.

  It is easy to see that if $\ca, \tau$ has growth constant $\gamma >0$ then
  $\ca, \tau$ maintains at least factorial growth.
  For let $0<\eps<1$ (for example take $\eps=\frac12$) and let $n_0$ be such that
  $(1-\eps)^n n! \gamma^n \leq \tau(\ca_n) \leq (1-\eps)^{-n} n! \gamma^n$
  for each $n \geq n_0$.
  Then for each $n > n_0$ and $1 \leq j \leq n-n_0$,
\[
  \tau(\ca_n) \geq (1-\eps)^n n! \gamma^n = (1-\eps)^n (n-j)! \gamma^{n-j}  (n)_j \gamma^j
  \geq (1-\eps)^{2n} \tau(\ca_{n-j}) (n)_j \gamma^j,
\]
  and~(\ref{eqn.mafg}) follows for each $1 \leq j <n$.
  
  Other graph classes may be shown to maintain at least factorial growth
  even though we may not know whether they have a growth constant, and indeed that is the point of introducing this property.
  We shall see below that
  this holds for example when the class is closed under subdividing edges.   Note in particular that for any given surface $S$,
  the class of graphs $G$ which have minimum degree at least 2 and are embeddable on $S$
  is closed under subdividing edges. 
  



\begin{lemma} \label{lem.subdiv}
  Let $\ca$ be a minor-closed class of graphs, and suppose that each graph in $\ca$ contains at least one edge such
  that the graph obtained by subdividing that edge is still in $\ca$.
  Then for each $\tau$, the weighted class $\ca, \tau$ maintains at least factorial growth.
\end{lemma}

\begin{proof} \hspace{.1in}
  Call an edge $e$ in a graph $G \in \ca$ 
  {\em arbitrarily subdividable} if $\ca$ contains each graph obtained by repeatedly subdividing $e$.
  Observe that each graph $G \in \ca$ must contain an arbitrarily subdividable edge.
  
  Let $1 \leq j < n$ be such that $\ca_{n-j}$ is non-empty. Pick a set $S$ of $j$ vertices
  from $[n]$, and list them as $v_1,v_2,\ldots,v_j$.
  Pick a graph $G$ in $\ca$ on $[n]\setminus S$ and an arbitrarily subdividable edge $e= uw$
  where $u<w$, and replace $e$ by the path $u v_1 v_2 \cdots v_j w$, to form $G'$.
  The weighted number of constructions is at least $\tau(\ca_{n-j}) \ (n)_j \eta^j$,
  where $\eta= \lambda$. 

  How often can a given graph $G'$ on $[n]$ be constructed?
  In $G'$ there must be a path $v_1,v_2,\ldots,v_j$ of vertices of
  degree 2, where $j$ is known.
  The number of such paths is at most $n^2$.
  But knowing the path determines $G$, $e$ and the list, so $G'$ can be
  constructed at most $n^2$ times.  Hence
\[
  \tau(\ca_n) \geq  \tau(\ca_{n-j}) (n)_j \eta^j \  g(n),
\]
  where   $g(n)=n^{-2}$, 
  and the lemma follows.
\end{proof}
  \medskip

  The above lemma shows that a weighted class of graphs embeddable on a given surface
  maintains at least factorial growth, as long as we exclude the edgeless graphs
  (and there is just one such for each order $n$).
  In general, to show that a class $\ca$ maintains at least factorial growth, it suffices to look at a subclass~$\cb$,
  as long as $\cb$ does not form too small a proportion of the graphs in $\ca$.

\begin{lemma} \label{lem.growth-subclass}
  Let the class $\ca$ of graphs contain a subclass $\cb$ such that
  $\tau(\cb_n) = (1+o(1))^n \tau(\ca_n)$ and $\cb, \tau$ maintains at least factorial growth.
  Then $\ca, \tau$ maintains at least factorial growth.
\end{lemma}

\begin{proof}
  Let $\eta>0$ and the function $g_1(n)=(1+o(1))^n$ be such that for each $n$ and
  each $j$ with $1 \leq j < n$ we have
\begin{equation} \label{eqn.cacb1}
  \tau(\cb_n) \geq  \tau(\cb_{n-j}) \ (n)_j \ \eta^j \ g_1(n).
\end{equation}
  Let $g_2(n) = \tau(\cb_n)/\tau(\ca_n)$, so that $g_2(n)=(1+o(1))^n$.
  Next let $g_3(n) = \min_{1 \leq j < n} g_2(n-j)$.  Then it is easy to check that
  $g_3(n)=(1+o(1))^n$. But for each $j$ with $1 \leq j < n$,
\begin{equation} \label{eqn.cacb2}
  \tau(\cb_{n-j})=g_2(n-j) \tau(\ca_{n-j}) \geq g_3(n) \tau(\ca_{n-j}).
\end{equation}
  Now $\tau(\ca_n) \geq \tau(\cb_n)$, and so by~(\ref{eqn.cacb1}) and~(\ref{eqn.cacb2})
\[
  \tau(\ca_n) \geq \tau(\ca_{n-j}) \ (n)_j \ \eta^j \ g_4(n)
\]
  where $g_4(n) = g_1(n) g_3(n) =(1+o(1))^n$.
\end{proof}
  \medskip
  
  Finally here let us check that a seemingly more `local' and weaker condition implies that
  $\ca$ maintains at least factorial growth (as defined at~(\ref{eqn.mafg2})).
  
\begin{lemma} \label{lem.malfgnod}
  Assume that $\ca_n \neq \emptyset$ for infinitely many $n$.
  Let $0<\delta \leq 1$, $0<\eta\leq 1$ and $g(n)=(1+o(1))^n$; and
  suppose that 
\[ \tau(\ca_n) \geq  \tau(\ca_{n-j}) \ (n)_j \ \eta^j \ g(n) \]
   for each sufficiently large $n$ and each $1 \leq j \leq \delta n$. Then
  $\ca$ maintains at least factorial growth (with the same $\eta$).
\end{lemma}

\begin{proof}
  Let $0<\eps \leq 1$. 
  Let $N_0$ be sufficiently large that $g(n) \geq (1-\delta \eps/2)^n$ for all $n \geq N_0$.
  Let $N \geq N_0$ be such that $\ca_{N} \neq \emptyset$.
  Let us first show that
\begin{equation} \label{eqn.maint1}
  \tau(\ca_n) \geq  (n)_j \ \eta^j e^{-\eps n} \cdot \tau(\ca_{n-j}) 
\end{equation}
  for each $n >N$ and  $j=1,\ldots,n-N$.
   
  %
  Let $n>N$ and $1 \leq j \leq n-N$.
  Let $n_i= \lfloor (1-\delta)^i n \rfloor$ for $i=0,1,2,\ldots$.  Let $i_0$ be the least $i$ such that
  $n_i \leq n-j$, and redefine $n_{i_0}$ as $n-j$.
  Observe that $\sum_{i=0}^{i_0-1} n_i \leq n \sum_{i \geq 0} (1-\delta)^i = n/\delta$.
  Note also that 
  $(1-x/2) \geq e^{-x}$ for $0<x \leq 1$ and so
  $(1-\delta \eps/2)^{1/\delta} \geq e^{-\eps}$.
  Hence
\[ \prod_{i=0}^{i_0-1} g(n_i) \geq (1-\delta \eps/2)^{\sum_{i=0}^{i_0-1} n_i} \geq (1-\delta \eps/2)^{n/\delta}
  \geq e^{-\eps n}. \]
    Now
\begin{eqnarray*}
  \tau(\ca_n)
    & \geq &
  \prod_{i=0}^{i_0-1} (n_i)_{(n_i - n_{i+1})} \eta^{n_i - n_{i+1}} g(n_i) \cdot \tau(\ca_{n_{i_0}})\\
     & = &
  (n)_{j}\ \eta^{j} \cdot \prod_{i=0}^{i_0-1} g(n_i) \cdot  \tau(\ca_{n-j})\\
     & \geq &
  (n)_{j}\ \eta^{j} e^{-\eps n} \cdot  \tau(\ca_{n-j}),
\end{eqnarray*}
   as required, and we have proved~(\ref{eqn.maint1}).

  Let $\alpha = \max_{1 \leq i \leq N} \{ \tau(\ca_i)/\tau(\ca_{N}) \}$.  Let $c = (\alpha N!)^{-1}$,
  and note that $0<c\leq 1$.  We want to establish~(\ref{eqn.mafg2}).  It will suffice to show that for each $n>N$,
\begin{equation} \label{eqn.maint2}
  \tau(\ca_n) \geq  \tau(\ca_{n-j})\, (n)_j \ \eta^j e^{-\eps n} \cdot c
\end{equation}
  for each $n >N$ and  $j=1,\ldots,n-1$.
  Since $c \leq 1$ this is immediate from~(\ref{eqn.maint1}) for $j \leq n-N$.  So assume that $n-N<j <n$,
  and let $k=j-(n-N)$ so $1 \leq k \leq N-1$.   
  By~(\ref{eqn.maint1}) with $j=n-N$,
\[ \tau(\ca_n) \geq  \tau(\ca_{N})\, (n)_{n-N} \ \eta^{n-N} e^{-\eps n}. \] 
  But
\[ \tau(\ca_N) \geq \tau(\ca_{N-k}) \cdot \alpha^{-1} \geq \tau(\ca_{N-k}) (N)_k \eta^k \cdot c. \]
  Putting the last two inequalities together gives 
\[  \tau(\ca_n) \geq  \tau(\ca_{n-j}) \, (n)_j \ \eta^j e^{-\eps n} \cdot c \]
  as required.   
\end{proof}

\subsection{Being well-behaved: proof of Lemma~\ref{lem.wellb}}
\label{subsec.l2.1}

  If $\ca$ is an addable proper minor-closed class of graphs then $\ca$ is decomposable,
  and $\ca,\tau$ has a growth constant by lemma~\ref{lem.addable-gc};
  and if $\ca$ is $\gs$, or $\ca$ is the class of graphs which contain at most $k$ vertex-disjoint cycles,
  then $\ca$ is closed under subdividing edges, and 
  $\ca,\tau$ has a growth constant. 
  Also in each case $\ca$ is dichotomous, as we saw in Section~\ref{subsec.dichotomous},
  and so $\ca,\tau$ is very well-behaved.

  It remains to show that each very well-behaved weighted graph class $\ca,\tau$ is well-behaved.
  %
  If $\ca$ is closed under subdividing edges then by Lemma~\ref{lem.subdiv}
  $\ca, \tau$ maintains at least factorial growth, and we are done. 
  So suppose now that $\ca$ is decomposable.
  Then $\ca$ is addable, and so also the class $\ca^{\delta \geq 2}$ of graphs in $\ca$
  with minimum degree at least 2 is addable, if it is non-empty.
  Recall that $\cb$ denotes the class of connected graphs in $\ca^{\delta \geq 2}$.
  By Lemma~\ref{lem.addable-gc}, if $\ca^{\delta \geq 2}$ is non-empty then $\ca^{\delta \geq 2},\tau$ has a
  growth constant, and so by lemma~\ref{lem.tauconn} $\cb,\tau$ has a growth constant (the same one);
  and a fortiori $\cb,\tau$ maintains at least factorial growth, as noted in Section~\ref{subsec.maint}.
  The only remaining case is when $\ca^{\delta \geq 2}$ is empty, and so we are done.


\subsection{$\rho(\ca,\tau)$ as $\tau$ varies}
\label{subsec.tau-vary}


  Consider a proper minor-closed class $\ca$ of graphs, with $\rho(\ca,(1,1))$ finite.
  For example if $\ca$ contains the class $\cf$ of forests then $\rho(\ca,(1,1)) \leq \rho(\cf,(1,1)) = e^{-1} < \infty$.
  We saw that the radius of convergence $\rho(\ca,\tau)$ does not depend on the component parameter $\nu$.
  Let us write $\rho(\lambda)$ for $\rho(\ca,\tau)$.
  
  Let $c=c(\ca)$ be the maximum average degree of a graph in $\ca$, and recall that $c$ is finite
  (see the proof of Lemma~\ref{lem.addable-gc}).
  Note that $c \geq 2$: for if $c<2$ and $\cb$ is a class of all graphs with average degree at most $c$,
  then
\[ |\cb_n| = \sum_{0 \leq m \leq cn/2} {{n \choose 2} \choose m} \leq n\left( \frac{en}{c} \right)^{cn/2}
   = e^{(c/2) n \log n},\]
  so $(|\cb_n| / n!)^{1/n} = o(1)$, that is $\rho(\cb,\lambda)=0$.  (See~\cite{gmsw05} for related details.)

\begin{prop} \label{prop.mc-tauvary}
  The function $\rho(\ca,\lambda)$ is continuous and strictly decreasing on $(0,\infty)$.
  If $\; 0<\lambda \leq 1$ then
  \[\rho(\ca,1)\, \lambda^{-1} \leq \rho(\ca,\lambda) \leq \rho(\ca,{1}) \lambda^{-c/2}, \]
  and if $\: \lambda \geq 1$ then
  \[\rho(\ca,{1}) \lambda^{-c/2} \leq \rho(\ca,\lambda) \leq \rho(\ca,{1}) \lambda^{-1}.\]
\end{prop}
   For the class $\cf$ of forests we have $c=2$, and
\[ \rho(\cf,\lambda) = \rho(\cf,{\bf 1}) \lambda^{-1} = (e \lambda)^{-1}\]
  as we already noted; and we see that the inequalites above for $\rho(\ca,\lambda)$ are tight.
  \smallskip
  
\begin{proof}
  To show $\rho(\ca,1)\, \lambda^{-1} \leq \rho(\ca,\lambda)$ for $0<\lambda \leq 1$.
  Let $0<\eps<1$.  Let $\ca' = \{ G \in \ca: e(G) \leq (1-\eps) v(G)\}$,
  and let $\ca'' = \ca \setminus \ca'$. Then
  $\rho(\ca',\lambda)= \infty$ as we saw above, and $\rho(\ca'',\lambda) = \rho(\ca,\lambda)$.
  Then for $\lambda \leq 1$
\[ \tau(\ca_n'') \leq | \ca_n''| \, {\lambda}^{(1-\eps)n} \leq | \ca_n| \, {\lambda}^{(1-\eps)n},\]
  and so  $\rho(\ca, \lambda) = \rho(\ca'', \lambda) \geq \rho(\ca,{\bf 1}) \, {\lambda}^{-(1-\eps)}$.
  But this holds for each $0<\eps<1$ so
  $\rho(\ca, \lambda) \geq \rho(\ca,{\bf 1}) \hat{\lambda}^{-1}$.
  (We needed no assumptions on $\ca$ here.)
  \smallskip
  
  To show $\rho(\ca,\lambda) \leq \rho(\ca,{1}) \lambda^{-1}$ for $\lambda \geq 1$.
  As we saw following~(\ref{eqn.exp}), $\rho(\cc,\lambda)= \rho(\ca,\lambda)$, where $\cc$ is the class of
  connected graphs in $\ca$.  Thus for $\lambda \geq 1$, $\tau(\cc_n) \geq |\cc_n| {\lambda}^{n-1}$, so
  $\rho(\cc,\lambda) \leq \rho(\cc,{\bf 1})/{\lambda}$ and hence
  $\rho(\ca,\lambda) \leq \rho(\ca,{\bf 1})/{\lambda}$.
  
  For the remaining inequalities, observe that $\tau(\ca_n) = \sum_{G \in \ca_n} \lambda^{e(G)}$.
  Thus if ${\lambda} \leq 1$, $\tau(\ca_n) \geq |\ca_n| {\lambda}^{cn/2}$, and so
  $\rho(\ca,\lambda) \leq \rho(\ca,{\bf 1}) {\lambda}^{-c/2}$;
  and if ${\lambda} \geq 1$, $\tau(\ca_n) \leq |\ca_n| {\lambda}^{cn/2}$, and so
  $\rho(\ca,\lambda) \geq \rho(\ca,{\bf 1}) {\lambda}^{-c/2}$.  
  \smallskip

  To show that $\rho(\ca,\lambda)$ is strictly decreasing on $(0,\infty)$, consider $\ca''$ from the first part of the proof,
  with say $\eps=\frac12$.  Let $\eta>0$.  Then
\[ \sum_{G \in \ca_n''} ((1+\eta) \lambda)^{e(G)} \geq (1+\eta)^{ n/2} \cdot \sum_{G \in \ca_n''} \lambda^{e(G)}\]
  and so
\[ \rho(\ca,(1+\eta)\lambda) = \rho(\ca'',(1+\eta)\lambda) \leq (1+\eta)^{-\frac12} \rho(\ca,\lambda).\] 
  Now to show that $\rho(\ca,\lambda)$ is continuous on $(0,\infty)$, observe that
\[ \sum_{G \in \ca_n} ((1+\eta) \lambda)^{e(G)} \leq (1+\eta)^{cn/2} \cdot \sum_{G \in \ca_n} \lambda^{e(G)} \]
  and so
\[ \rho(\ca,(1+\eta)\lambda) \geq (1+\eta)^{-c/2} \rho(\ca,\lambda).\]
\end{proof}


  \section{Distribution of the Boltzmann random graph}
  \label{sec.proofU}

  Let us first check equation~(\ref{eqn.Agen}).
  Recall that the class $\ca$ of graphs is closed under isomorphism.
  We identify an unlabelled graph on $n$
  vertices with an equivalence class under graph isomorphism of graphs on vertex set $[n]$.
  Since each graph $H \in \cu \ca_n$ consists of $\frac{n!}{\aut(H)}$
  graphs $G \in \ca_n$, we have
  \[
  \frac{x^{v(H)} y^{e(H)} z^{\kappa(H)}}{\aut(H)} =
  \sum_{G \in H} \frac{\aut(H)}{v(H)!}  \frac{x^{v(H)} y^{e(H)} z^{\kappa(H)} }{\aut(H)}
  =  \sum_{G \in H} \frac{x^{v(G)} y^{e(G)} z^{\kappa(G)} }{v(G)!}.
  \]
  Thus
  \[
  A(x,y,z)
  = \sum_{H \in \cu \ca} \sum_{G \in H} \frac{x^{v(G)} y^{e(G)} z^{\kappa(G)} }{v(G)!}
  = \sum_{H \in \cu \ca} \frac{x^{v(H)}y^{e(H)}
  z^{\kappa(H)}}{\aut(H)},
  \]
  proving~(\ref{eqn.Agen}).
  \bigskip

  \begin{proofof} {Theorem~\ref{thm.U}}
  Each sum and product below is over all $H$ in $\cu\cc$.
  Let the unlabelled graph $G$ consist of $n_H$ components isomorphic to $H$ for each
  $H \in \cu \cc$, where $0 \leq \sum_{H} n_H < \infty$.
  Then
\[ \rho^{v(G)} = \prod_{H} \rho^{v(H) n_H}, \;\; \lambda^{e(G)} = \prod_{H} \lambda^{e(H) n_H},
  \;\; \nu^{\kappa(G)} = \prod_{H} \nu^{n_H} \]
  and
\[  \aut(G) = \prod_{H} \aut(H)^{n_H} {n_H}! .\]
  Hence
\[  \frac{ \rho^{v(G)} \lambda^{e(G)} \nu^{ \kappa(G)} }{\aut(G)} =   
  \prod_{H} \frac{\mu(H)^{n_H}}{n_H!}.\]
  
  Also since $\sum_{H} \mu(H) = C(\rho, \tau)$ by~(\ref{eqn.Agen}) applied to $\cc$,
  \[ \frac{1}{A(\rho,\tau)} =
    e^{-C(\rho,\tau)} = \prod_{H} e^{-\mu(H)}.
  \]
  Hence
  \begin{eqnarray*}
  \pr[R =G]
  &=& e^{-C(\rho,\tau)} \frac{ \rho^{v(G)} \lambda^{e(G)} \nu^{ \kappa(G)} }{\aut(G)}\\
  &=& \prod_{H} e^{-\mu(H)} \frac{\mu(H)^{n_H}}{n_H!}\\
  &=& \prod_{H} \pr[\Po(\mu(H))=n_H].
  \end{eqnarray*}
  Thus the probability factors appropriately, and
  the random variables $\kappa(R,H)$ for $H \in \cu\cc$ satisfy
  \[
  \pr[\kappa(R,H)=n_H \;\; \forall H \in \cu\cc] =
  \prod_H \pr[\kappa(R,H)=n_H].
  \]
  This holds for every choice of non-negative integers $n_H$ with
  $\sum_{H \in \cu\cc} n_H < \infty$, and
  thus also without this last restriction
  (since both sides are zero if the sum is infinite).
  This completes the proof of the theorem.
  \end{proofof}


\section{Smoothness and 2-core: proof of Theorem~\ref{thm.sumup}}
\label{sec.proofs}

  In this section we deduce Theorem~\ref{thm.sumup}, after two
  preliminary subsections:
  one on deducing smoothness for a class $\ca$ from smoothness for the
  class of connected graphs in $\ca$, and one on smoothness for
  classes of connected graphs.


\subsection{Smoothness: from connected to general}
\label{subsec.contogen}


\begin{lemma} \label{lem.smoothconv}
  Let the class $\ca$ of graphs be bridge-addable,
  and let $R_n \in_{\tau} \ca$.
  Let $\cf$ denote the class $\cf_{\ca}$ of graphs freely addable to~$\ca$,
  and suppose that \whp  $\Frag(R_n) \in \cf$. 
  Let $\cc$ be the class of connected graphs in $\ca$,
  let $\rho = \rho({\cal C},\tau)$ satisfy $0<\rho<\infty$, and
  suppose further that $\cc,\tau$ is smooth.
  Then $F(\rho,\tau)$ is finite, where $F$ is the exponential generating function for
  $\cf$;  and the weighted class $\ca,\tau$ is smooth.
\end{lemma}

\begin{proofof}{Lemma \ref{lem.smoothconv}}
  Recall that $F(x,\tau) = \sum_{j \geq 0} \tau(\cf_j) x^j /j!$.
  We shall show that $F(\rho,\tau)$ is finite and
  $\tau(\ca_n) \sim F(\rho,\tau) \tau(\cc_n)$,
  from which it will follow immediately that $\ca,\tau$ is smooth.
  Let $0<\eta<1$.  Then we are to show that $F(\rho,\tau)$ is finite,
  and that for $n$ sufficiently large
  \begin{equation} \label{eqn.ctoa}
  (1-\eta) F(\rho,\tau) \tau(\cc_n) \leq \tau(\ca_n) \leq (1+\eta) F(\rho,\tau) \tau(\cc_n).
  \end{equation}
  Let $\eps>0$ be sufficiently small that
  $\eps \leq 1/(e^{\nu/\lambda}+2)$, and
  $(1-2\eps)^{-1}(1+\eps) \leq 1+\eta$ and $(1-\eps)^2 \geq 1 - \eta$.
  By our assumptions and Lemma~\ref{lem.tauconn} (b),
  we may fix positive integers $k$ and $n_0$ sufficiently large that
  $\sum_{j=0}^{k} \tau(\cf_j) \rho^j/j!$ is at least
  $(1-\eps) F(\rho,\tau)$ if $F(\rho,\tau)$ is finite,
  and is at least $e^{\nu/\lambda}+2$ otherwise; and
  $\pr[\frag(R_n)>k] < \eps$ and
  $\pr[\Frag(R_n) \not\in \cf]< \eps$ for all $n \geq n_0$.

  Since $\cc,\tau$ is smooth, there exists $n_1 \geq n_0$ sufficiently large that
  for all $n \geq n_1$, the ratio
  $\tilde{r}_n = n \tau(\cc_{n-1})/\tau(\cc_{n})$ satisfies
  \[
  (1-\eps)^{1/k} \rho < \tilde{r}_n < (1+\eps)^{1/k} \rho.
  \]
  Then for each $n \geq n_1 +k$ and each $j=1,\ldots,k$,
  since
\[
  \frac{(n)_j \tau(\cc_{n-j})} {\tau(\cc_{n})} = \prod_{i=1}^j \tilde{r}_{n-i+1}
\]
  we have
\[
  (1-\eps) \rho^j <  \frac{(n)_j \tau(\cc_{n-j})}{\tau(\cc_{n})} < (1+\eps) \rho^j.
\]
  Denote  $\sum_{j=0}^{k} {n \choose j} \tau(\cf_j) \tau(\cc_{n-j})$  by $\tilde{a}_n$.
  Observe that $\tilde{a}_n \leq \tau(\ca_n)$: we shall see that $\tilde{a}_n$ is an approximation
  to $\tau(\ca_n)$.  Note that
\[
  \tilde{a}_n = \tau(\cc_{n}) \sum_{j=0}^{k}
  \frac{\tau(\cf_j)}{j!}
  \frac{(n)_j \tau(\cc_{n-j})}{\tau(\cc_{n})};
\]
  and so for each $n \geq n_1 +k$ we have
\[
  (1-\eps) \tau(\cc_n) \sum_{j=0}^{k}\frac{\tau(\cf_j) \rho^j}{j!}
  \leq \tilde{a}_n \leq
  (1+\eps) \tau(\cc_n) \sum_{j=0}^{k}\frac{\tau(\cf_j) \rho^j}{j!}.
\]

  We may now see that $F(\rho,\tau)$ is finite.
  For suppose not.  Then for each $n \geq n_1 +k$
\[
  \tau(\ca_n) \geq \tilde{a}_n
  \geq (1-\eps)(e^{\nu/\lambda}+2) \tau(\cc_n)
  \geq (e^{\nu/\lambda}+1) \tau(\cc_n).
\]
  But since $\ca$ is bridge-addable, by Lemma~\ref{lem.tauconn}
  the probability that $R_n$ is connected is at least $e^{-\nu/\lambda}$,
  and we obtain the contradiction that
\[
  e^{\nu/\lambda} \cdot \tau(\cc_n) \geq \tau(\ca_n)
  \geq (e^{\nu/\lambda}+1) \cdot \tau(\cc_n).
\]
  Hence $F(\rho,\tau)$ must be finite.

  From the above we have
  $\tilde{a}_n \leq (1+\eps) \tau(\cc_n) F(\rho,\tau)$,
  and
  $\tilde{a}_n \geq (1-\eps)^2 \tau(\cc_n) F(\rho,\tau)$.
  But
  \[
  \tau(\ca_n) = \tilde{a}_n + \sum_{G} \{\tau(G) : G \in \ca_n, \frag(G)>k \mbox{ or } \Frag(G) \not\in \cf \}.
  \]
  Thus $\tau(\ca_n) \leq \tilde{a}_n + 2 \eps \, \tau(\ca_n)$, and so
\[
    \tau(\ca_n) \leq (1- 2\eps)^{-1} \tilde{a}_n \leq (1+ \eta) \tau(\cc_n) F(\rho,\tau);
\]
  and
\[
  \tau(\ca_n) \geq \tilde{a}_n \geq (1-\eps)^2 \tau(\cc_n) F(\rho,\tau) \geq
  (1- \eta) \tau(\cc_n) F(\rho,\tau).
\]
  So~(\ref{eqn.ctoa}) holds and we are done.
\end{proofof}

  We need Lemma~\ref{lem.smoothconv} for graph classes which may not be decomposable
  (such as $\cg^S$) but let us note here an elegant general result for a decomposable class $\ca$
  and the corresponding class $\cc$ of connected graphs: 
  by Corollary 4.3 of Bell and Burris~\cite{bb03}, if $\cc, \tau$ is smooth then so is $\ca, \tau$.




\subsection{Connected graphs and smoothness}
\label{subsec.conn-smooth}

  Let us call a class $\ca$ of graphs {\em trimmable} if it satisfies
  $\; G \in \ca \Leftrightarrow \core(G) \in \ca$.  (To tell if a
  graph is in such a class, it does not matter if we repeatedly trim off leaves.)
  Recall that by convention the empty graph is in $\ca$, so if $\ca$ is trimmable then $\ca$ contains every forest.
  Observe also that a minor-closed class of graphs is trimmable if and only if
  each excluded minor $H$ has minimum degree $\delta(H) \geq 2$.
  Also, if $\ca$ is bridge-addable and monotone (that is, closed under forming subgraphs)
  then the class $\cc$ of connected graphs in $\ca$ is trimmable.

  Suppose that a non-empty class $\cc$ of connected 
  graphs is trimmable. Then from what we saw about trees we have
  $\liminf_n \left( \tau(\cc_n)/n! \right)^{1/n} \geq \lambda e$.  
  Let $\cb=\cc^{\delta \geq 2}$.  Then $C(x,\tau)=B(\nu^{-1} T^o(x,\tau),\tau)$,
  where $C$, $B$ and $T^o$ are the exponential generating functions for $\cc$, $\cb$ and the rooted trees, respectively.
  Now if $0<x< (\lambda e)^{-1}$ then $\nu^{-1} T^o(x,\tau) < \lambda^{-1}$
  by the last result in Section~\ref{subsec.connresults}.
  Hence if $\rho(\cb,\tau) \geq \lambda^{-1}$ then $\rho(\cc, \tau) \geq (\lambda e)^{-1}$,
  and so $\cc, \tau$ has growth constant $\lambda e$.
  We extend this observation below.
  

%

 \begin{lemma} \label{lem.ctrim1}
   Let the non-empty class $\cc$ of connected graphs be trimmable, and
   let $\cb = \cc^{\delta \geq 2}$. 
   Suppose that either (a) $\cb, \tau$ has growth constant $\beta$ and $\beta > \lambda$, 
   in which case we let $\gamma = \beta e^{\lambda/\beta}$ (which is $>\lambda e$)
   and let $\alpha=1- \lambda/\beta$;
   or (b)  $\cb, \tau$ has radius of convergence $\geq \lambda^{-1}$
   in which case we let $\gamma = \lambda e$ and let $\alpha =0$.
   
   Then $\cc, \tau$ is smooth, with growth constant~$\gamma$.
   Further, 
   let $R^{\cc}_n \in_{\tau} \cc$: then for any $\eps>0$
   \[
   \pr[|v(\core(R^{\cc}_n))- \alpha n| > \eps n] = e^{-\Omega(n)}.
   \]
\end{lemma}

\begin{proof}
  If $\cb$ is empty then $\cc$ is the class $\ct$ of all trees,
  and we saw in Section~\ref{sec.comps} above
  that $\ct, \tau$ is smooth with growth constant $\lambda e$
  (and $\beta=0, \gamma= \lambda e$ and $\alpha =0$).
  Thus we may assume that $\cb$ is non-empty.
  For $3 \leq k \leq n$ let
  \[ f(n,k)=  \nu \sum \left\{ \lambda^{e(G)}: G \in \cc_n, v(\core(G))=k \right\}.\]
  Observe that $|\cc_1|=1$ if the one-vertex graph $K_1$ is in $\cc$,
  and $|\cc_1|=0$ otherwise.
  %
  The main idea of the proof was inspired by~\cite{bcr08},
  and goes as follows for case (a), when $\cb,\tau$ has growth constant $\gamma>\lambda e$. 
  (We consider the case (b) later.)
  We first show that
  \[ f(n,k) = (1+o(1))^n \ n! \ \gamma^n \]
  when $k = (\alpha +o(1)) n$,
  and the expression on the right side gives an asymptotic approximation for $\tau(\cc_n)$.
  Further, the dominant contribution in the sum
  \begin{equation} \label{eqn.csum}
    \tau(\cc_n) = \sum_{k=3}^{n} f(n,k) + |\cc_1| \ n^{n-2} \lambda^{n-1} \nu
  \end{equation}
  is from $k$ as above; for all such $k$
  \[ \frac{f(n+1,k)}{(n+1) f(n,k)} \sim \gamma ; \]
  and this yields
  \[ \frac{\tau(\cc_{n+1})}{(n+1) \tau(\cc_n)} \sim \gamma.\]

  Now for the details.
  Let us not yet assume that $\cb$ has a growth constant (so that we can re-use the argument later).
  Recall that the number of forests on $[n]$ consisting of $k$ trees
  where vertices $1,\ldots,k$ are all in different trees is $kn^{n-1-k}$,
  see for example Theorem 3.3 of~\cite{moon70}.
  Thus
  \begin{equation} \label{eqn.fnk}
  f(n,k) = {n \choose k} \ \tau(\cb_k) \ \lambda k(\lambda n)^{n-1-k}.
  \end{equation}
  Of course $f(n,n)= \tau(\cb_n)$.
  We aim next to prove the three results
  (\ref{eqn.beta1}), (\ref{eqn.beta2}) and~(\ref{eqn.beta3}) below.
  We first consider case (a),
  and then case (b) will follow easily.

  Suppose then that $\beta>0$. 
  Let $r(n) = \frac{\tau(\cb_n)}{n! \beta^n}$
  (so that we will have $r(n) = (1+o(1))^n$ below once we assume that $\cb, \tau$ has growth constant $\beta$).
  %
  Let $s(n) = \frac{n^n}{n! e^n}$, so that by Stirling's formula $s(n) \sim (2 \pi
  n)^{-\frac12}$, and $s(n) \leq 1$ for all $n$.
  Then for $3 \leq k \leq n-1$, writing $\kappa = k/n$,
\begin{eqnarray*}
  f(n,k)  & = &
  n! \ \frac{\tau(\cb_k) }{k!} \ \frac{\lambda k}{\lambda n} \ \left(\frac{\lambda n}{n-k}\right)^{n-k}
   \ \frac{(n-k)^{n-k}}{(n-k)!}\\
  & = &
  n! \ r(k) \ \beta^k \ \frac{k}{n} \left(\frac{\lambda e}{1-k/n}\right)^{n-k} s(n-k)\\
  & = &
  n! \ \beta^n \
  \left(\frac{\lambda e}{\beta(1-\kappa)}\right)^{(1-\kappa)n} \cdot
  \kappa \, r(k) \, s(n-k)\\
  & = &
  n! \ \beta^{n} \  \left(h(1-\kappa)\right)^{n} \cdot \kappa \, r(k) \, s(n-k)
\end{eqnarray*}
  where $h(x) = (\frac{\lambda e}{\beta x})^{x}$ for $x>0$ and
  $h(0) =1$.
  Thus (without yet assuming that $\cb, \tau$ has a growth constant) we have
\begin{equation} \label{eqn.pregc}
  f(n,k) = n! \ \beta^{n}  \left(h(1-\kappa)\right)^{n} \cdot
  \kappa \, r(k)\, s(n-k).
\end{equation}
  Note that the function $h(x)$ strictly increases up to $x=\lambda/\beta$,
  where it has value $e^{\lambda/\beta}$, and strictly decreases above $\lambda/\beta$.
  Hence
\[f(n,k) \leq \ n! \ \beta^{n}  \left(e^{\lambda/\beta}\right)^{n} r(k)\]
  and recalling that $\gamma= \beta e^{\lambda/\beta}$ we have
\begin{equation} \label{eqn.pregcbound}
  f(n,k) \leq n! \, \gamma^n r(k).
\end{equation}
  (We will use this inequality in the proof of lemma~\ref{lem.ctrim2}.)

  Now assume that $\cb, \tau$ has growth constant $\beta$, so that
  $g(n) = (1+o(1))^n$ by the definition of the growth constant.
  %
  Then from~(\ref{eqn.pregc}),
  uniformly over $k$ with $3 \leq k \leq n$, still writing
  $\kappa = k/n$, we have
\begin{equation} \label{eqn.pregcapprox}
    f(n,k) = (1+o(1))^n \ n! \ \beta^{n}  \left(h(1-\kappa)\right)^{n}.
\end{equation}
  %
  We need to consider two subcases.

  (i) Suppose first that $\beta>\lambda$, so
  $\gamma = \beta e^{\lambda/\beta}$ and $\alpha=1-\lambda/\beta$.
  Then it follows from~(\ref{eqn.csum}), (\ref{eqn.pregcapprox}) and the properties of $h$ that
  \begin{equation} \label{eqn.beta1}
    \tau(\cc_n) = (1+o(1))^n \ n! \ \gamma^n
  \end{equation}
  (the possible term $n^{n-2}\lambda^{n-1} \nu$ in the sum~(\ref{eqn.csum}) is negligible
  since $\gamma > \lambda e$),
  and for any $\delta>0$ there exists $\eta>0$ such that
  \begin{equation} \label{eqn.beta2}
    \sum_{k:|k-\alpha n| \geq \delta n} f(n,k)
    \leq (1-\eta +o(1))^n \ n! \ \gamma^n.
  \end{equation}
  Thus for any $\delta>0$
\begin{equation} \label{eqn.beta3}
  \frac{\tau(\cc_{n+1})}{(n+1) \tau(\cc_{n})} \sim
  \frac{\sum_{k:|k- \alpha n| < \delta n} f(n+1,k) }{\sum_{k:|k- \alpha n| < \delta n} (n+1) f(n,k)}.
\end{equation}

  (ii) 
  If $\beta = \lambda$ then $\gamma= \lambda e$ and $\alpha=0$,
  and the results~(\ref{eqn.beta1}), (\ref{eqn.beta2}) and~(\ref{eqn.beta3})
  follow as above.
  \smallskip


   Finally consider the case $\rho(\cb,\tau) \geq \lambda^{-1}$.
   We shall see that we have exactly the same results~(\ref{eqn.beta1}), (\ref{eqn.beta2})
   and~(\ref{eqn.beta3}) as for the case (ii). 
   We know that $\tau(\cc_n) \geq (1+o(1))^n n! (\lambda e)^n$.
   Also, we may add connected graphs to $\cc$, maintaining trimmability,
   to form $\cc'$ so that if $\cb'$ denotes $\{G \in \cc':\delta(G) \geq 2\}$ then
   $\cb', \tau$ has growth constant $\lambda$. 
   Thus from (\ref{eqn.beta1}) and~(\ref{eqn.beta2}) for $\cc'$ and $\cb'$
   we obtain the corresponding results for $\cc$ and $\cb$ in this case,
   and then we may deduce~(\ref{eqn.beta3}).  We have now established
  (\ref{eqn.beta1}), (\ref{eqn.beta2}) and~(\ref{eqn.beta3}) for both cases (a) and (b).

  By equation~(\ref{eqn.fnk}), for each $k$ such that $3 \leq k \leq n$ and
  $\cb_k \neq \emptyset$, writing $k = \kappa n$ we have
  \[
    \frac{f(n+1,k)}{(n+1)f(n,k)} = \frac{\lambda n}{n+1-k} \ (1+\frac1{n})^{n-k}
    = \frac{\lambda(1+\frac1{n})^{(1-\kappa)n}}{1+ \frac1{n} - \kappa}.
  \]
  %
  Now $\lambda \frac{e^{1-\alpha}}{1-\alpha} = \beta e^{\lambda/\beta}= \gamma$ if $\beta>\lambda$,
  and $\lambda \frac{e^{1-\alpha}}{1-\alpha} = \lambda e = \gamma$ if $\beta =\lambda$.
  Let $\eps>0$.  By considering the two cases for $\beta$,
  we see that there exist $n_0$ and $\delta>0$ such that whenever
  $n \geq n_0$ and $|\kappa-\alpha| < \delta$ we have
  \[ (1-\eps) \gamma \leq
  \frac{\lambda (1+\frac1{n})^{(1-\kappa)n}}{1+ \frac1{n} - \kappa} \leq
   (1+\eps) \gamma. \]
  Hence by~(\ref{eqn.beta3}) we have
  \[ (1-\eps +o(1)) \gamma \leq
  \frac{\tau(\cc_{n+1})}{(n+1) \tau(\cc_n)} \leq
  (1+\eps +o(1)) \gamma.\]
   Thus
   \[ \frac{\tau(\cc_{n+1})}{(n+1) \tau(\cc_n)} \to  \gamma \;\;
   \mbox{ as } n \to \infty\]
   as required.
   The last part of the lemma, concerning the size of the core,
   follows directly from~(\ref{eqn.beta1}) and~(\ref{eqn.beta2}).
\end{proof}

  \medskip


   

\begin{lemma} \label{lem.ctrim2}
  Let the class $\cc$ of connected graphs be trimmable, and
  suppose that $\cc,\tau$ has growth constant $\gamma$;
  and let $\cb = \cc^{\delta \geq 2}$. 
   
  If $\gamma> \lambda e$ and $\cb, \tau$ maintains at least factorial growth
  then $\cb,\tau$ has growth constant $\beta$ where $\beta$ is the unique root $>\lambda$ of
  $\ \beta e^{\lambda/\beta} = \gamma$.
  If $\gamma \leq \lambda e$ then $\rho(\cb,\tau) \geq \beta^{-1}$ where $\beta=\lambda$. 
\end{lemma}

\begin{proof}
  We prove first that $\limsup (\tau(\cb_n)/n!)^{1/n} \leq \beta$,
  by showing that otherwise equation~(\ref{eqn.pregc}) in the proof of the last result
  will yield $\limsup (\tau(\cc_n)/n!)^{1/n} > \gamma$.
  Let $\hat{\beta}>\beta$.
  If $\tau(\cb_k) \geq k! \hat{\beta}^k$ and we let
  $n = \lceil\frac{k}{1-(\lambda/\hat{\beta})} \rceil$ then
  $h(1-k/n) \sim h(\lambda/\hat{\beta}) = e^{\lambda/\hat{\beta}}$, and so by~(\ref{eqn.pregc})
  \[
  \left(\frac{\tau(\cc_n)}{n!}\right)^{\frac{1}{n}} \geq
  \left(\frac{f(n,k)}{n!}\right)^{\frac{1}{n}} \geq (1+o(1)) \hat{\beta} e^{\lambda/\hat{\beta}}.
  \]
  But the function $f(x)=x \ln \lambda/x$ is strictly increasing for $x>\lambda$, so
  $\hat{\beta} e^{\lambda/\hat{\beta}}  > \beta e^{\lambda/\beta} = \gamma$.
  Thus $\limsup (\tau(\cc_n)/n!)^{1/n} > \gamma$,
  which contradicts the assumption that $\cc,\tau$ has growth constant $\gamma$.

  For the case when $\gamma> \lambda e$ we also need a lower bound.
  %
  Let $0<\eps<1$.  We want to show that for all sufficiently large $n$ we have
  \begin{equation} \label{eqn.claim1}
   \tau(\cb_n) \geq n! \ \beta^n (1-\eps)^{n}.
  \end{equation}
  We now use the assumption that $\cb,\tau$ maintains at least factorial growth
  (in the form in Lemma~\ref{lem.malfgnod}).
  Let $\delta'>0$, $\eta>0$ and $g(n)=(1+o(1))^n$ be such that for each $n$ and
  each $j$ with $1 \leq j \leq \delta' n$ we have
  \[ \tau(\cb_n) \geq  \tau(\cb_{n-j}) \ (n)_j \ \eta^j \ g(n).  \]
  Let $0<\delta<\min\{\delta',1- (\lambda/\beta) \}$
  be sufficiently small that
  $(\eta/\beta)^{\delta} \geq 1- \eps/3$.
  We claim that there is an $n_0$ such that for all $n \geq n_0$
  there is an $\tilde{n}$ with
  $|n-\tilde{n}| < \delta n$ such that
  \begin{equation} \label{eqn.tildan}
    \tau(\cb_{\tilde{n}}) \geq \tilde{n}! \ \beta^{\tilde{n}}
    (1- \eps/3)^{\tilde{n}}.
  \end{equation}
  The idea is that the proof of the last result, Lemma~\ref{lem.ctrim1},
  shows that there must be such an $n_0$ since otherwise $\tau(\cc_n)$ would be too small for each large~$n$.
  For suppose there is no such~$n_0$.  Then for arbitrarily large values of
  $k$, each $j$ with $|j-k| \leq \delta k$ has
  $\tau(\cb_j) < j! \beta^j (1- \eps/3)^j$; that is, $r(j) < (1- \eps/3)^j$,
  where $r(j) = \frac{\tau(\cb_j)}{j! \beta^j}$ as in the proof of the last lemma.  
  Consider such a $k$, and let $n = \lceil\frac{k}{1-(\lambda/\beta)} \rceil$
  as above.  As we saw in~(\ref{eqn.beta2}) above,
  there is a constant $\eta>0$ such that
\[
  \sum_{j:|j-k| \geq \delta k} f(n,j) \leq n! \, \gamma^n \, (1-\eta +o(1))^n,
\] 
  and now also by~(\ref{eqn.pregcbound})
\[
  \sum_{j:|j-k| < \delta k} f(n,j) \leq n! \, \gamma^n \sum_{j:|j-k| < \delta k} r(j)
  \leq c \cdot n! \, \gamma^n (1-\eps/3)^{(1-(\lambda/\beta)-\delta)n}.
\] 
  for a suitable constant $c$.
  Let $\eta'>0$ satisfy $\eta'<\eta$ and
  $1-\eta' > (1-\eps/3)^{1-(\lambda/\beta)-\delta}$.
  Then by~(\ref{eqn.csum}) and the above
  \[ \left(\frac{\tau(\cc_n)}{n!}\right)^{\frac{1}{n}}
  \leq (1-\eta') \ \gamma \]
  if $n$ is sufficiently large.  This contradicts the
  assumption that $\cc, \tau$ has growth constant $\gamma$, and completes the proof of~(\ref{eqn.tildan}).
  Indeed we can insist that there is a value $\tilde{n}$ as above with
  $\tilde{n} \leq n \leq (1+ \delta) \tilde{n}$.  To see this we may
  apply to $n -\lfloor \delta n/2 \rfloor$ the current version of~(\ref{eqn.tildan})
  with $\delta$ replaced by $\delta/2$.

  Let $n_1 \geq 2 n_0$ be such that $g(n) \geq (1-\eps/3)^n$ for all $n \geq n_1$.
  Let $n \geq n_1$. It will suffice for us to show that~(\ref{eqn.claim1})
  holds for $n$. Let $j=n-\tilde{n}$. If $j=0$ there is nothing to prove so we may
  assume that $1 \leq j \leq \delta n$. Then
  \begin{eqnarray*}
  \tau(\cb_n) & \geq & \tau(\cb_{\tilde{n}}) \ (n)_j \ \eta^j \ g(n)\\
  & \geq &
   n! \beta^n \beta^{-j} (1- \eps/3)^n \eta^j \ g(n)\\
  & \geq &
   n! \beta^n (\eta/\beta)^{j} (1- \eps/3)^{2n}\\
  & \geq &
   n! \beta^n (1- \eps/3)^{3n} \; \geq \;
   n! \beta^n (1- \eps)^n,
  \end{eqnarray*}
  as required.
  %
%
%
\end{proof}


\subsection{Proof of Theorem~\ref{thm.sumup}}
\label{subsec.sumup-proof}

  The class $\cc$ has the same growth constant $\gamma$ as $\ca$,
  for example since $\ca$ is bridge-addable. 
  If $\gamma > \lambda e$ then, since $\cc$ is trimmable,
  by Lemma~\ref{lem.ctrim2} $\cb = \cc^{\delta \geq 2}$ has growth constant $\beta$.
  If $\gamma \leq \lambda e$ then $\rho(\cb,\tau) \geq \lambda^{-1}$.
  But now (without restriction on $\gamma$)
  Lemma~\ref{lem.ctrim1} shows that $\cc$ is smooth, and further shows
  that, for $R^{\cc}_n \in_{\tau} \cc$, we have for any $\eps>0$ that
\begin{equation} \label{eqn.cc-core}
  \pr(|v(\core(R^{\cc}_n))-\alpha n| > \eps n) = e^{-\Omega(n)}.
\end{equation}
  Also, for $R_n^{\ca} \in_{\tau} \ca$, $\pr(\Frag(R_n^{\ca}) \in \cf_{\ca}) = 1-o(1)$ by Lemma~\ref{lem.Frag}.
  We may now use Lemma~\ref{lem.smoothconv} to show that
  $\ca$ is smooth.  At this point we have proved~(\ref{eqn.cc-core})
  and parts (a) and (b) of Theorem~\ref{thm.sumup}.


  Next we prove part (c).  
  Observe that conditional on $\bigc(R_n^{\ca})=n'$ the distribution of
  $\Bigc(R_n^{\ca})$ is the same as that of $R_{n'}^{\cc}$.  Thus for each $j <n$ and each $t$
\begin{eqnarray*}
  && \pr(v(\core(R_{n-j}^{\cc})) \geq t)\\
  & \leq &
  \pr(v(\core(R_{n}^{\ca})) \geq t \, | \, \frag(R_{n}^{\ca})=j) \leq
  \pr(v(\core(R_{n-j}^{\cc})) \geq t-j).
\end{eqnarray*}
  Let $\eps>0$.  Let $\omega=\omega(n) = \lfloor \eps n/2 \rfloor$.  Then
\begin{eqnarray*}
  &&\pr(v(\core(R_{n}^{\ca})) \geq (\alpha +\eps)n)\\
  & \leq &
  \pr(( v(\core(R_{n}^{\ca})) \geq (\alpha +\eps)n) \cap (\frag(R_{n}^{\ca}) \leq \omega))
  + \pr(\frag(R_{n}^{\ca}) > \omega).
\end{eqnarray*}
  By Lemma~\ref{lem.tauconn} (b)  
  we know that the second term
  $\pr(\frag(R_{n}^{\ca}) > \omega)$ is $o(1)$.  But the first term equals
\begin{eqnarray*}
  &&
  \sum_{j=0}^{\omega} \pr\left( (v(\core(R_{n}^{\ca})) \geq (\alpha +\eps)n)
   \cap ( \frag(R_{n}^{\ca} ) =j) \right)\\
  & \leq &
  \sum_{j=0}^{\omega} \pr\left( v(\core(R_{n-j}^{\cc})) \geq (\alpha +\eps)n -j\right)
  \;\; =  e^{-\Omega(n)}.
\end{eqnarray*}
  Thus
\[ \pr(v(\core(R_{n}^{\ca})) \geq (\alpha +\eps)n) = o(1). \]
  Similarly
\begin{eqnarray*}
  &&\pr(v(\core(R_{n}^{\ca})) \leq (\alpha -\eps)n)\\
  & \leq &
  \sum_{j=0}^{\omega} \pr\left( v(\core(R_{n-j}^{\cc})) \leq (\alpha -\eps)n \right)
   + \pr(\frag(R_{n}^{\ca}) > \omega) \;\; = \; o(1).
\end{eqnarray*}

  Now consider the remaining part of the theorem, part (d), and assume that $\gamma>\lambda e$.
  Note first that is it very unlikely that $\core(R_n)$ is empty; for this happens (if and) only if $R_n$ is a forest,
  and the probability of this happening is $(\lambda e /\gamma +\!o(1))^n = e^{-\Omega(n)}$.
  Also, the probability that $\Bigc(R_n)$ is a tree is $e^{-\Omega(n)}$.  For
\begin{eqnarray*}
 && \pr(\Bigc(R_n) \mbox{ is a tree and } \bigc(R_n) \geq \frac23 n)\\
 & \leq & 
   \sum_{\frac23 n \leq a \leq n} {n \choose a} \frac{\tau(\ct_a) \cdot \tau(\ca_{n-a})}{\tau(\ca_n)}
  = 
   \sum_{\frac23 n \leq a \leq n} \frac{ \frac{\tau(\ct_a)}{a!} \cdot \frac{\tau(\ca_{n-a})}{(n-a)!} }
   {\frac{\tau(\ca_n)}{n!} }\\
 &=&
   (1+o(1))^n \sum_{\frac23 n \leq a \leq n} \frac{ (\lambda e)^a \gamma^{n-a} }{\gamma^n}
   = (1+o(1))^n (\lambda e / \gamma)^{\frac23 n}  = e^{-\Omega(n)}.
\end{eqnarray*}
  But if $\core(R_n)$ is non-empty and $\Bigc(R_n)$ is not a tree,
  then $\core(R_n)$ is connected if and only if $\Frag(R_n)$ is acyclic. Thus
\[ | \pr(\core(R_n) \mbox{ connected (and $\neq \emptyset$)})- \pr(\Frag(R_n) \mbox{ acyclic})| = e^{-\Omega(n)}.\]  
   Finally, the probability that $\Frag(R^{\ca}_n)$ has no non-tree components tends to
   $e^{-(D(\rho,\tau) - T(\rho,\tau))}$ by Corollary~\ref{cor.comps} (b) applied to $\cd \setminus \ct$.
    





\section{Poisson convergence}
\label{sec.poisson-conv}

  Let $\tau=(\lambda, \nu)$ and $\rho>0$ be given. As in~(\ref{eqn.mudef}), we use the notation
\[
  \mu(H)=\rho^{v(H)} \lambda^{e(H)} \nu^{\kappa(H)}/\aut(H) \;\;\; \mbox{ for each graph } H .
\]
  Also, let $r_n=n \tau(\ca_{n-1})/\tau(\ca_n)$,
  and assume that $\ca_n \neq \emptyset$ when necessary.
  Further, recall the notation $(n)_k = n(n-1) \cdots (n-k+1)$.
  
  The following lemma is a slight extension 
  for example of Lemma 4.1 of~\cite{cmcd09}.
  It will be a key result for taking advantage of smoothness.
  Given a graph $G$ and a connected graph $H$ we let $\kappa(G,H)$
  be the number of components of $G$ isomorphic to $H$.
\begin{lemma} \label{lem.pconv1}
  Let $\ca$ be any class of graphs, and let $\tau$ 
  and $\rho>0$ be given. Let $H_{1}, \ldots, H_{h}$ be pairwise
  non-isomorphic connected graphs, each freely addable to $\ca$.
  Let $k_{1}, \ldots, k_{h}$ be non-negative integers, and let
  $K = \sum_{i = 1}^{h} k_{i} v(H_i)$.  Then
  for $R_{n} \in_{\tau} \ca$
  \[\E\left[\prod_{i=1}^h \left( \kappa(R_n,H_i) \right)_{k_{i}}\right] =
  \prod_{i=1}^h \mu(H_i)^{k_i} \cdot
  \prod_{j=1}^{K} (r_{n-j+1}/\rho).\]
\end{lemma}

\begin{proof}
  We may construct a graph $G$ in $\mathcal{A}_{n}$ with
  $\kappa(G,H_i) \geq k_i$ for each $i$
  as follows:
  choose a list of $K$ vertices; put a copy of $H_1$ on the first $v(H_1)$ vertices in the list,
  if $k_1>1$ put another copy of $H_1$ on the next $v(H_1)$ vertices,
  and so on until we put a copy of $H_h$ on the last $v(H_h)$ vertices in the list;
  and finally put any graph of order $n-K$ in $\ca$
  on the remaining $n-K$ vertices.
  %
  The sum over all such constructions of the weight $\lambda^{e(G)} \nu^{\kappa(G)}$
  of the graph $G$ constructed is
\[
  (n)_K \prod_{i=1}^{h}\left( \aut(H_i)^{-1}
  \lambda^{e(H_i)} \nu \right)^{k_i} \cdot  \tau(\ca_{n-K}).
\]
  Now observe that each graph $G \in \ca_n$ is constructed exactly
  $\prod_{i = 1}^{h}(\kappa(G,H_{i}))_{k_i}$ times; and so the above expression
  equals
\[  \sum_{G \in \ca_n} \prod_{i=1}^h (\kappa(G,H_i))_{k_{i}} \lambda^{e(G)} \nu^{\kappa(G)}.\]
  But by definition $\E\left[\prod_{i=1}^m (\kappa(R_{n},H_i))_{k_{i}}\right]$
  is $\tau(\ca_n)^{-1}$ times this last quantity. 
  Hence
\begin{eqnarray*}
  \E\left[\prod_{i=1}^m (\kappa(R_{n},H_i))_{k_{i}}\right]
  &=& (n)_K \prod_{i=1}^{h}\left( \aut(H_i)^{-1}
  \lambda^{e(H_i)} \nu \right)^{k_i} \cdot
  \tau(\ca_{n-K})/\tau(\ca_n)\\
  &=& \prod_{i=1}^{h} \mu(H_i)^{k_i} \cdot
  \prod_{j=1}^K \left(\rho^{-1} (n-j+1)
  \frac{\tau(\ca_{n-j})}{\tau(\ca_{n-j+1})}\right)\\
  &=& \prod_{i=1}^{h} \mu(H_i)^{k_i} \cdot
  \prod_{j=1}^K (r_{n-j+1}/\rho)
\end{eqnarray*}
  as required.
  \end{proof}
  \bigskip

 \noindent
  When we add the assumption that $\ca, \tau$ is smooth,
  we find convergence of distributions.
  Recall that $\rho(\ca,\tau)$ denotes the radius of convergence of $A(x,\tau)$ as a power series in~$x$.

\begin{lemma} \label{lem.conv2}
  Let the weighted graph class $\ca, \tau$ be smooth,
  and let $\rho = \rho({\ca, \tau})$.
  Let $H_1,\ldots,H_h$ be a fixed family of pairwise non-isomorphic
  connected graphs, each freely addable to $\ca$.
  Then as $n \to \infty$ the joint distribution of the random
  variables
  $\kappa(R_n,H_1)$, $\ldots,$ $\kappa(R_n,H_h)$
  converges in total variation to the product distribution
  $\Po(\mu(H_1)) \otimes \cdots \otimes \Po(\mu(H_h))$.
\end{lemma}

\begin{proof}
  Since $r_n \to \rho$ as $n \to \infty$, by the last lemma
\[
  \E\left[\prod_{i=1}^h
  \left(\kappa(R_n,H_i) \right)_{k_{i}}\right] \to
  \prod_{i=1}^{h} \mu(H_i)^{k_i}
\]
  as $n \to \infty$, for all non-negative integers
  $k_1,\ldots,k_h$. A standard result on the Poisson distribution
  now shows that the joint distribution of the random variables
  $\kappa(R_n,H_1),\ldots,\kappa(R_n,H_h)$ tends to that of
  independent random variables
  $\Po(\tau(H_1)),\ldots, \Po(\tau(H_h))$, see for example Theorem~6.10 of
  Janson, {\L}uczak and Ruci{\'n}ski~\cite{jlr00}.
  Thus for each $h$-tuple of non-negative integers
  $(t_1,\ldots,t_h)$
\[
  \pr[\kappa(R_n,H_i)=t_i \; \forall i] \to \prod_i
  \pr[\kappa(R_n,H_i)=t_i] \;\; \mbox{ as } n \to \infty;
\]
  and so we have pointwise convergence of probabilities,
  which is equivalent to convergence in total variation.
\end{proof}


\section{$\Frag(R_n)$ and connectivity}
\label{sec.frag-conn}

  The following lemma parallels Lemma~\ref{lem.smoothconv}, which showed that, under suitable conditions,
  if the class $\cc$ of connected graphs in $\ca$ is smooth then the class $\ca$ is
  smooth.
  The lemma below shows the converse result that if $\ca$ is smooth
  then, for $R_n \in_{\tau} \ca$ the probability that $R_n$ is connected tends to a limit, and so $\cc$ is smooth.
\begin{lemma} \label{lem.stillgen1}
  Let the graph class $\ca$ be
   bridge-addable; let $\rho = \rho({\ca,\tau})$;
  let $\cf_{\ca}$ denote the class of graphs freely addable to~$\ca$;
  and suppose that, for $R_n \in_{\tau} \ca$, \whp $\Frag(R_n) \in \cf_{\ca}$.
  Let $\cc$ and $\cd$ be the classes of connected graphs in $\ca$ and $\cf_{\ca}$ respectively,
  and let $F_{\ca}$ denote the exponential generating function of $\cf_{\ca}$.
  Suppose that $\ca,\tau$ is smooth.
  Then $F_{\ca}(\rho,\tau)$ and $D(\rho,\tau)$ are finite; $\kappa(R_n) \to_{TV} 1+ \Po(D(\rho,\tau))$;
  and in particular
  \[
   \Pr[R_n\mbox{ is connected }] \;\to\;  e^{-D(\rho,\tau)} = 1/F_{\ca}(\rho,\tau) \;\; \mbox{ as } n \to
   \infty,
  \]
  and so $\cc,\tau$ is smooth.
\end{lemma}

\begin{proof}
  We follow the method of proof of Lemma 4.3 of~\cite{cmcd09}.
  We first show that $D(\rho,\tau)$ is finite.
  By Lemma~\ref{lem.tauconn} (b) we may choose a (fixed)
  $k$ sufficiently large that $\pr[\frag(R_n) \geq k] \leq \frac13$ for all $n$,
  and $\pr[\Po(2k) \geq k] \geq \frac23$.
  Suppose that $D(\rho,\tau) \geq 2k+1$.
  Then by~(\ref{eqn.Agen}) there are distinct $H_1,\ldots,H_m$ in $\cu \cd$ such
  that $\sum_{i=1}^{m} \mu(H_i) = \mu_0 \geq 2k$.
  It follows by Lemma~\ref{lem.conv2} that,
  for $n>k \max_{i} v(H_i)$
  \[
  \frac13 \geq
  \pr[\frag(R_n) \geq k]  \geq
  \pr[\sum_{i=1}^{m} \kappa(R_n,H_i) \geq k]
  \to \pr[\Po(\mu_0) \geq k] \geq \frac23
  \]
  as $n \to \infty$, a contradiction.
  Hence $D(\rho,\tau)$ is finite, and so $F_{\ca}(\rho,\tau)$ is finite too.

  Now let $\mu = D(\rho,\tau)$.
  Let $k$ be a fixed positive integer and let $\eps >0$.
  We want to show that for $n$ sufficiently large we have
  \begin{equation} \label{eqn.poisson1}
  |\pr[\kappa(\Frag(R_n))=k] - \pr[\Po(\mu)=k]| < \eps.
  \end{equation}
  By our assumptions, there is an $n_0$ such that for each $n \geq n_0$
  \begin{equation} \label{eqn.k1}
  \pr[\frag(R_n) > n_0] + \pr[ \Frag(R_n) \not\in \cf_{\ca} ] < \eps/3.
  \end{equation}
  List the graphs in ${\cal U} \cd$
  in non-decreasing order of the number of vertices as $H_1,H_2,\ldots$.
  For each positive integer $m$ let
  $\mu^{(m)} = \sum_{i=1}^{m} \mu(H_i)$.
  Note that $D(\rho,\tau)= \sum_{H \in \cu \cd} \mu(H)$ by~(\ref{eqn.Agen}) applied to $\cd$.
  Thus we may choose $n_1 \geq n_0$
  such that, if $m$ is the largest index such that $v(H_m) \leq n_1$, then
  \begin{equation} \label{eqn.k2}
  |\pr[\Po(\mu)=k] - \pr[\Po(\mu^{(m)})=k]| < \eps /3.
  \end{equation}
  Observe that for any graph $G$ with more than $2n_1$ vertices,
  if $\frag(G) \leq n_0$ and $\Frag(G) \in \cf_{\ca}$,
  then $\kappa(\Frag(G))$ is the number of components of $G$ isomorphic
  to one of $H_1,\ldots,H_m$ (that is, with order at most $n_1$).
  Let $X_n$ denote the number of components of $R_n$ isomorphic
  to one of $H_1,\ldots,H_m$.
  Let $n > 2n_1$.  Then
  \begin{equation} \label{eqn.k5}
  |\pr[\kappa(\Frag(R_n))\!=\!k]\!-\!\pr[X_n\!=\!k]| \leq
  \pr[\frag(R_n)\!>\!n_0] + \pr[\Frag(R_n)\!\not\in\!\cf_{\ca}] < \eps /3.
  \end{equation}
  But by Lemma~\ref{lem.conv2}, for $n$ sufficiently large,
  \[ |\pr[X_n=k] - \pr[\Po(\mu^{(m)})=k]| < \eps/3,\]
  and then by~(\ref{eqn.k2}) and~(\ref{eqn.k5}) the
  inequality~(\ref{eqn.poisson1}) follows.
  Thus we have shown that $\kappa(R_n) \to_{TV} 1+ \Po(D(\rho,\tau))$,
  and in particular
  \begin{equation} \label{eqn.conn0}
  \tau(\cc_n)/\tau(\ca_n) = \pr(R_n \mbox{ is connected }) \to e^{-D(\rho,\tau)} \mbox{ as } n \to \infty.
  \end{equation}
  Finally observe that since $\ca,\tau$ is smooth, and
  $\tau(\cc_n)/\tau(\ca_n)$ tends to a non-zero limit as $ n \to \infty$
  (namely $e^{-D(\rho,\tau)}$),
  it follows that $\cc,\tau$ is smooth.
\end{proof}
  \medskip

  \noindent
  The next lemma has similar premises to 
  Lemma~\ref{lem.stillgen1}, and obtains further conclusions.  We use the same notation.

\begin{lemma} \label{lem.stillgen2}
  Let $\ca$ be 
  bridge-addable;
  let $\rho = \rho({\ca,\tau})$,
  and suppose that
  $\Frag(R_n) \in \cf_{\ca}$ \whp where $R_n \in_{\tau} \ca$.
  Let $\cc$ be the class of connected graphs in $\ca$.
  Assume that either $\ca,\tau$ or $\cc,\tau$ is smooth.
  Then both $\ca,\tau$ and $\cc,\tau$ are smooth; 
  $F_{\ca}(\rho,\tau)$ is finite; and the unlabelled graph $F_n$ corresponding to $\Frag(R_n)$ satisfies $F_n \to_{TV} F$, where
  \[
    \pr[F=H] = \frac{\mu(H)}{F_{\ca}(\rho,\tau)} \;\; \mbox{ for each } H \in \cu \cf_{\ca}.
  \]
\end{lemma}

\begin{proof}
  Lemmas~\ref{lem.smoothconv} and~\ref{lem.stillgen1} show that both $\ca,\tau$ and $\cc,\tau$ are smooth,
  and that $F_{\ca}(\rho,\tau)$ is finite.
  Let $a_n = \tau(\ca_n)$ and $c_n =\tau(\cc_n)$.
  Let $\cd$ be the class of connected graphs in $\cf_{\ca}$.
  By Lemma~\ref{lem.stillgen1}
  \begin{equation} \label{eqn.conn}
  c_n/a_n \to e^{-D(\rho,\tau)} = 1/F_{\ca}(\rho,\tau) \mbox{ as } n \to \infty.
  \end{equation}
  Given a graph $G$ on a finite subset $V$ of the positive integers
  let $\phi(G)$ be the natural copy of $G$ moved down on to $\{1,\ldots,|V|\}$; that is,
  let $\phi(G)$ be the graph on $\{1,\ldots,|V|\}$ such that the
  increasing bijection between $V$ and $\{1,\ldots,|V|\}$ is an
  isomorphism between $G$ and $\phi(G)$.

  Let $H$ be any graph in $\cb$ on $[h]$.  Then
\begin{eqnarray*}
  \pr[\phi(\Frag(R_n))=H] & = & {n \choose h} \ \frac{c_{n-h}}{a_n}
  \;\; = \;\;  \frac{c_{n-h}}{a_{n-h}} \ \frac{1}{h!} \ \frac{(n)_h a_{n-h}}{a_n}\\
  & = & \frac{c_{n-h}}{a_{n-h}} \ \frac{1}{h!} \ \prod_{i=0}^{h-1} r_{n-i}
  \;\; \to \;\; e^{-D(\rho,\tau)} \frac{\rho^h}{h!}
\end{eqnarray*}
  as $n \to \infty$ by~(\ref{eqn.conn}) and the fact that
  $\ca,\tau$ is smooth.
  Now by symmetry
  \[ \pr [F_n \cong H ] = \frac{h!}{\aut(H)} \
  \pr[\phi(\Frag(R_n))= H]\]
  and hence as $n \to \infty$
  \[\pr[F_n \cong H] \to e^{-D(\rho,\tau)}
  \frac{\rho^h}{\aut(H)} = \Pr(F \cong H).\]
  Thus for each $H \in \cu \cb$, as $n \to \infty$ we have  $\pr[F_n = H] \to \pr(F=H)$; that is,
  $F_n \to_{TV} U$.
\end{proof}

  \bigskip

  We need one last lemma to complete the proof of Theorem~\ref{thm.Frag}.

  \begin{lemma} \label{lem.mc2}
  Let the weighted graph class $\ca,\tau$ be well-behaved;
  let $\cd$ be the class of connected graphs in $\cf_{\ca}$,
  with generating function $D$; and let $\rho = \rho({\ca,\tau})$.
  Then 
  $D'(\rho,\tau)$ is finite
  (where we are differentiating with respect to the first variable).
  \end{lemma}

  \begin{proof}
  Note that $0<\rho<\infty$.
  By Lemma~\ref{lem.tauconn} (b), 
  $\E[\frag(R_n)] \leq c $ for all~$n$ where $c=2\nu/\lambda$.
  Suppose that $D'(\rho,\tau) \geq (c+3)/\rho$.
  Then by~(\ref{eqn.Agen}) applied to $\cd$, there are distinct $H_1,\ldots,H_m$ in $\cu \cd$ such
  that $\sum_{i=1}^{m} v(H_i) \ \mu(H_i) = \alpha \geq c+2$.
  Let $n_0 = \max_i v(H_i)$.  Then
  \[
  \E[\frag(R_n)] \geq \E\left[ \sum_{i=1}^{m} v(H_i)\ \kappa(R_n,H_i)\right] - n_0 \pr[\bigc(R_n) \leq n_0].
  \]
  Now $\ca,\tau$ is smooth by Theorem~\ref{thm.sumup} (a).  Thus as $n \to \infty$
  \[
  \E\left[ \sum_{i=1}^{m} v(H_i) \kappa(R_n,H_i)\right] \to \alpha
  \]
  by Lemma~\ref{lem.conv2}, and by Lemma~\ref{lem.tauconn} 
  \[
  \pr[\bigc(R_n) \leq n_0] \leq \pr[\kappa(R_n) \geq n/n_0] \leq \pr[\Po(\frac{\nu}{\lambda}) \geq n/n_0 -1] = o(1).
  \]
  Hence
  $\E[\frag(R_n)] \geq \alpha -o(1) \geq c+1$ for $n$ sufficiently large, contradicting our choice of $c$.
  \end{proof}
  \bigskip

  \begin{proofof}{Theorem~\ref{thm.Frag}}
  By Lemma~\ref{lem.Frag}, \whp $\Frag(R_n) \in \cf_{\ca}$.
  Hence by Lemma~\ref{lem.stillgen2}, $\ca,\tau$ and $\cc,\tau$ are smooth
  (this also follows from Theorem~\ref{thm.sumup} (a)) and $F_n \to_{TV} R$ as required.
  Finally note that $\E[v(F)]$ is $\rho \, D'(\rho,\tau)$, which is finite by the last lemma.
  %
  \end{proofof}
  \bigskip

 \begin{proofof}{Corollary \ref{cor.comps}}
   The only thing that does not follow directly from the fact that $F_n \to_{TV} F$
   is the convergence of the moments in part (b) (which yields the results on moments in part~(c)).
   But $0 \leq \kappa(F_n,\cd) \leq \kappa(R_n) -1 \leq \Po(\frac{\nu}{\lambda})$
   in distribution by Lemma~\ref{lem.tauconn},
   and so convergence for the $j$th moment follows from convergence in total variation.
  \end{proofof}


\section{Proof of appearances results}
\label{sec.apps-proofs}

  Theorem~\ref{thm.appearances1} and Proposition~\ref{thm.apps2}
  may be proved along the lines of the corresponding proofs
  in~\cite{cmcd08}, but for completeness we give proofs here.
  \medskip


\begin{proofof}{Theorem~\ref{thm.appearances1}}
  Let $\alpha =  \lambda \lambda^{e(H)}/(2e^2 \gamma^h (h+2)h!)$.
  We shall prove that there exists $n_0$ such that
\begin{equation} \label{eqna}
  \Pr[f_H(R_n) \leq \alpha n] < e^{-\alpha n} \;\;\;
  \mbox{ for all } \; n \geq n_0.
\end{equation}
  We often write $x$ instead of $\lfloor x \rfloor$ or
  $\lceil x \rceil$ to avoid cluttering up formulae: this should cause the reader no problems.
  Since $\alpha>0$, $2^{\alpha} \geq (1+ \eps)^4$ for some $0<\eps< \frac12$.  Note that
\begin{equation}\label{eq:num-pg-true2}
   (1-\eps)(1+\eps)^2>1.
\end{equation}
  Let $g(n)$ denote $\tau(\ca_n)$.
  Since $\ca, \tau$ has growth constant~$\gamma$, there is a positive integer $n_0$ such
  that for each $n \geq n_0$ we have
\begin{equation}\label{eq:num-pg-true3}
  (1-\eps)^n \cdot n! \ \gamma^n \le g(n) \le (1+\eps)^n \cdot n! \ \gamma^n.
\end{equation}

  Let $\cb$ denote the class of graphs in $\ca$ such that $f_H(G) \leq \alpha \ v(G)$.
  Assume that equation~(\ref{eqna}) does not hold for some $n\ge n_0$;
  that is, assume that $\tau(\cb_n) \geq e^{-\alpha n} g(n)$.
   Let $\delta = \alpha h$.
  We shall show that
\[
  g((1+\delta)n) > (1+\eps)^{(1+\delta)n} \cdot [(1+\delta)n]! \cdot
  \gamma^{(1+\delta)n},
\]
  which will contradict~(\ref{eq:num-pg-true3}) and complete the proof of the theorem.

  In order to establish this inequality, we construct graphs $G'$ in $\ca$ on vertex set
  $\{1,\ldots,(1+\delta)n\}$ as follows.
  First we choose a subset of $\delta n$ special vertices
  (${(1+\delta)n \choose \delta n}$ choices) and a graph $G \in \cb$ on the
  remaining $n$ vertices. By assumption
\[
  \tau(\cb_n) \geq e^{-\alpha n}\cdot g(n) \ge e^{-\alpha n} (1-\eps)^n \gamma^n n!.
\]
  Next we consider the $\delta n$ special vertices.
  We partition them into $\alpha n$ (unordered) blocks of size $h$.
  On each block $B$ we put a copy of $H$ such that the increasing bijection from
  $\{1,\ldots,h\}$ to $B$ is an isomorphism between $H$ and this copy.
  Call the lowest numbered vertex in $B$ the {\em root} $r_B$ of the block.
  For each block $B$ we choose a non-special vertex $v_B$ and add the edge $r_B v_B$
  between the root and this vertex: observe that $H$ appears at $B$ in $G'$.
  This completes the construction of $G'$: note that $G \in \ca$ since $H$ is freely
  attachable to $\ca$.  For each choice of special vertices,
  the weight of constructions is
\begin{eqnarray*}
  && \tau(\cb_n) \cdot {\delta n \choose h \cdots h} \cdot \frac{1}{(\alpha n)!}
  \cdot n^{\alpha n} (\lambda \lambda^{e(H)})^{\alpha n}\\
  & = & \tau(\cb_n) \cdot \frac{(\delta n)! n^{\alpha n}}{(h!)^{\alpha n} (\alpha n)!} (\lambda \lambda^{e(H)})^{\alpha n}\\
  & \geq & \tau(\cb_n) \cdot (\delta n)! \ ( h! \alpha)^{-\alpha n}
  (\lambda \lambda^{e(H)})^{\alpha n}.
\end{eqnarray*}

  How often is the same graph $G'$ constructed?
  Call an oriented edge $e=uv$ {\em good} in $G'$ if it is a cut-edge in $G'$,
  the component $\tilde{G}$ of $G' - e$ containing $u$ has $h$ nodes,
  $u$ is the least of these nodes,
  and the increasing map from $\{1,\ldots,h\}$ to $V(\tilde{G})$ is an isomorphism
  between $H$ and $\tilde{G}$.
  Observe that each added oriented edge $r_B v_B$ is good.
  Indeed, there is exactly one good oriented edge for each appearance of $H$ in $G$.
  We shall see that $G'$ contains at most $(h+2) \alpha n$ good oriented edges.
  It will then follow that the number of times that $G'$ can be constructed is at most
  ${(h+2)\alpha n \choose \alpha n} \leq ((h+2)e)^{\alpha n}$.

  We may bound the number of good edges in $G'$ as follows.
  (a) There are exactly $\alpha n$ added oriented edges $r_B v_B$.
  (b) There are at most $\alpha n$ good oriented edges $e=uv$ in $E(G)$
  (that is, such that the unoriented edge is in $G$):
  for in this case the entire component of $G' - e$ containing $u$ must be contained in $G$
  (if it contained any other vertex it would have more than $h$ vertices),
  and so the number of them is at most $f_H(G)$.
  (c) There are at most $h \alpha n$ `extra' good oriented edges.
  To see this, consider a block $B$, and let $\tilde{H}$ denote the connected graph formed from the
  induced subgraph $G'[B]$ (which is isomorphic to $H$) together with the vertex $v_B$
  and the   edge $r_B v_B$.
  Each `extra' good oriented edge must be a cut edge in such a graph $\tilde{H}$
  oriented away from $v_B$, and in each graph $\tilde{H}$ there are at most $h$ cut-edges.

  We may put the above results together to obtain
\begin{eqnarray*}
 && g((1+\delta)n)\\
 & \ge &
    {(1+\delta)n \choose \delta n} \cdot
    e^{-\alpha n} (1-\eps)^n \gamma^n n! \cdot
    (\delta n)! \ (h! \alpha)^{-\alpha n} (\lambda \lambda^{e(H)})^{\alpha n} \cdot ((h+2)e)^{-\alpha n}\\
 & = &
    \left((1+\delta)n \right)! \cdot \gamma^{(1+\delta)n} \cdot (1-\eps)^n  \cdot
    2^{\alpha n} \\
 & \geq &
    g((1+\delta)n) \ (1+\eps)^{-(1+\delta)n}  \cdot (1-\eps)^n \cdot (1+ \eps)^{4n} \\
 & \geq & g((1+\delta)n) \ ((1-\eps)(1+\eps)^2)^n \;\; > \; g((1+\delta)n),
\end{eqnarray*}
  which is the desired contradiction.
\end{proofof}
\bigskip


  \begin{proofof}{Proposition~\ref{thm.apps2}}
  Denote $v(H)$ by $h$ and $\gamma^{-1}$ by $\rho$.
  Observe first that
  \[ \E[X_n(H)] = {n \choose h} (n-h) \lambda \lambda^{e(H)}
  \frac{\tau(\ca_{n-h})}{\tau(\ca_{n})}  \sim \lambda n \frac{\rho^{h} \lambda^{e(H)}}{h!}.\]
  Now consider $\E[(X_n(H))_2]$.
  For each graph $G$ on $\{1,\ldots,n\}$ let $Y_1(G,H)$ be the
  number of ordered pairs of appearances in $G$ of $H$ with disjoint
  vertex sets and such that the roots are not adjacent; and let
  $Y_2(G,H)$ be the number of ordered pairs of appearances in $G$ of $H$ such that
  either the vertex sets meet or the roots are adjacent.  Thus
  $\left(X_n(H)\right)_2 = Y_1(R_n,H) + Y_2(R_n,H)$.
  Now
  \[
  \E[ Y_1(R_n,H) ] =  \frac{(n)_{2h}}{(h!)^2} (n-2h)^2 \lambda^2 \lambda^{2e(H)}
  \frac{\tau(\ca_{n-2h})}{\tau(\ca_{n})}  \sim
  \left(\lambda n \frac{\rho^{h} \lambda^{e(H)}}{h!} \right)^2.
  \]
  But a graph $G$ of order at most $2h$ either consists of two
  appearances of $H$ with adjacent roots
  (and then $G$ has exactly two appearances of $H$), or the number of
  appearances of $H$ is at most the number of bridges in $G$.
  Thus $Y_2(G,H)$ is at most $2h$ times the number of components
  of $G$ of order at most $2h$, which is at most $2h \kappa(G)$;
  and so
  \[
  \E[ Y_2(R_n,H) ] \leq 2h \E[\kappa(R_n)] \leq
  2h(1+\frac{\nu}{\lambda})
  \]
  since $\E[\kappa(R_n)] \leq 1+\frac{\nu}{\lambda}$ by
  Lemma~\ref{lem.tauconn}.
  Hence
  \[ \E[(X_n(H))_2] = \E[ Y_1(R_n,H) ] + O(1) \sim
  \left(\lambda n \frac{\rho^{h} \lambda^{e(H)}}{h!} \right)^2. 
  \]
  Thus the variance of $X_n(H)$ is $o(\E[(X_n(H))^2])$, and
  the result follows by Chebyshev's inequality.
  \end{proofof}
  \medskip

   Finally consider the remark concerning disjoint pendant appearances
   following Proposition~\ref{thm.apps2}.
   By the above proof, it suffices to note that in any graph $G$ the number
   of pendant appearances of $H$ that share a vertex or edge with some other
   pendant appearance is at most $Y_2(G,H)$, and so
   $\E[\tilde{X}_n(H)] \leq \E[Y_2(R_n,H)] = O(1)$.





\section{Concluding remarks}
\label{sec.concluding}

  Sometimes it may be helpful to generalise our probability model one step further. 
  Bridges play a major role in this work.
  Recall that a {\em bridge} in a graph $G$ is an edge $e$ such that $G-e$ has one more component than $G$.
  For a graph $G$ we let $e_0(G)$ be the number of bridges in $G$ (the 0 is since a bridge is in 0 cycles)
  and let $e_1(G)=e(G)-e_0(G)$.
  In the definition of $\tau(G)$ let us replace $\lambda^{e(G)}$ by $\lambda_0^{e_0(G)} \lambda_1^{e_1(G)}$,
  where $\lambda_0$ and $\lambda_1$ are
  the {\em edge-parameters}. 

  Thus the distribution of our random graph is as follows.
  Let $\lambda_0>0$, $\lambda_1>0$ and $\nu>0$,
  let $\lambda=(\lambda_0,\lambda_1)$, and let 
  $\tau=(\lambda,\nu)$.  For each graph $G$ we let 
  $\lambda^{e(G)}$ denote $\lambda_0^{e_0(G)} \lambda_1^{e_1(G)}$,
  and let $\tau(G) = \lambda^{e(G)} \nu^{\kappa(G)}$.
  Now we proceed as before, and let $\pr(R_n=G) \propto \tau(G)$ for each $G \in \ca_n$.
  The most natural and interesting case is when $\lambda_0=\lambda_1$ but we learn more about the role of bridges by
  allowing the edge-parameters to differ.
  
  The results and proofs above change in a predictable way.  We simply replace $\lambda$ by $\lambda_0$,
  except when $\lambda$ appears as $\lambda^{e(G)}$, which we now interpret as
  $\lambda_0^{e_0(G)} \lambda_1^{e_1(G)}$.   This holds even for Proposition~\ref{thm.apps2},
  where $\lambda \cdot \lambda^{e(H)}$ becomes $\lambda_0 \cdot \lambda_0^{e_0(H)} \lambda_1^{e_1(H)}$.
  
  There are two places where the change is most apparent.
  Theorem~\ref{thm.sumup} is the upside, where the role of bridges is brought out: each
  $\lambda$ is replaced by $\lambda_0$, and in particular everything depends on how
  $\gamma$ compares to $\lambda_0 e$.
  Lemma~\ref{lem.wellb} is the downside: we noted that the previous proofs that the classes in parts (b) and (c)
  had growth constants in the uniform case extended easily to yield growth constants in the weighted case,
  but that holds only when $\lambda_0=\lambda_1$. Can we drop this extra condition?

  In the addable minor-closed case we could easily introduce more edge-weights,
  though it is not clear how much more we would learn.  For example given a graph $G$ and an edge $e=uv$ in $G$,
  we could let $f(e)$ be the maximum number of edge-disjoint paths between $u$ and $v$ in $G - e$;
  let $e_k(G)$ be the number of edges $e$ in $G$ with $f(e)=k$;
  and let $\lambda^{e(G)}$ mean $\prod_{k \geq 0} \lambda_k^{e_k(G)}$, where each parameter $\lambda_k>0$.
  Then the results above for an addable minor-closed class still hold, with the same proofs
   -- and perhaps we do learn something?
  Indeed, we could go as far as the very general model in~\cite{cmcd12},
  as long as we ensure that $\log \tau(G) = O(v(G))$. 

\bigskip

\noindent
{\bf Acknowledgement}  I would like to thank Kerstin Weller for helpful comments.


\end{document}